\documentclass[11pt, a4paper]{amsart}

\title{On distinct finite covers of $3$-manifolds}
\author[S.\ Friedl]{Stefan Friedl}
\address{Universit\"at Regensburg, Fakult\"at f\"ur Mathematik, Regensburg, Germany}
\email{sfriedl@gmail.com}
\urladdr{http://www.uni-regensburg.de/Fakultaeten/nat\_Fak\_I/friedl/}

\author[J.\ Park]{JungHwan Park}
\address{School of Mathematics, Georgia Institute of Technology, Atlanta, GA, USA}
\email{junghwan.park@math.gatech.edu  }
\urladdr{http://people.math.gatech.edu/~jpark929/}

\author[B.\ Petri]{Bram Petri}
\address{Institut de Math\'ematiques de Jussieu -- Paris Rive Gauche, Sorbonne Universit\'e, Paris, France}
\email{bram.petri@imj-prg.fr}
\urladdr{http://webusers.imj-prg.fr/~bram.petri/}

\author[J.\ Raimbault]{Jean Raimbault}
\address{Institut de Math\'ematiques de Toulouse, UMR 5219, Universit\'e Paul Sabatier--CNRS, 31062 Toulouse Cedex 9, France }
\email{Jean.Raimbault@math.univ-toulouse.fr}
\urladdr{https://www.math.univ-toulouse.fr/~jraimbau/}

\author[A.\ Ray]{Arunima Ray}
\address{Max-Planck-Institut f\"{u}r Mathematik, Bonn, Germany}
\email{aruray@mpim-bonn.mpg.de }
\urladdr{http://people.mpim-bonn.mpg.de/aruray/}

\date{\today}


\usepackage{amsmath,amsfonts,amssymb,amsthm,graphicx,overpic,hyperref,pinlabel}
\usepackage{float,enumitem, xypic, tikz, tikz-cd}
\usepackage[toc,page]{appendix}
\usetikzlibrary{arrows,matrix,positioning}
\usepackage[font=small]{caption}

\newtheorem{theorem}{Theorem}[section]
\newtheorem{proposition}[theorem]{Proposition}

\newtheorem{lemma}[theorem]{Lemma}
\theoremstyle{definition}
\newtheorem{definition}[theorem]{Definition}

\newtheorem{innerprprep}{Proposition}
\newenvironment{prprep}[1]
  {\renewcommand\theinnerprprep{#1}\innerprprep}
  {\endinnerprprep}

\newtheorem*{maintheorem}{Main Theorem}

\usepackage[margin=1in]{geometry}
\newcommand{\nc}{\newcommand}
\nc{\dmo}{\DeclareMathOperator}

\newcommand{\Z}{\mathbb{Z}}
\dmo{\Int}{Int}
\dmo{\coker}{Coker}
\nc{\abs}[1]{\left| #1 \right|}
\nc{\bigO}[1]{\mathcal{O}\left(#1\right)}
\nc{\card}[1]{\left|#1\right|}
\nc{\ceil}[1]{\left\lceil #1 \right\rceil}
\nc{\CC}{\mathbb{C}}
\nc{\floor}[1]{\left\lfloor #1 \right\rfloor}
\nc{\ZZ}{\mathbb{Z}}
\nc{\len}[1]{\left| #1 \right|}
\nc{\littleo}[1]{o\left(#1\right)}
\dmo{\Mat}{Mat}
\nc{\NN}{\mathbb{N}}
\nc{\norm}[1]{\left|\left| #1 \right|\right|}
\nc{\QQ}{\mathbb{Q}}
\nc{\RR}{\mathbb{R}}
\nc{\st}[2]{\left\{ #1 ;\; #2\right\}}
\dmo{\supp}{supp}
\nc{\tr}[1]{\mathrm{tr}\left(#1\right)}

\dmo{\area}{area}
\dmo{\conv}{conv}
\dmo{\diam}{diam}
\dmo{\dist}{\mathrm{d}}
\nc{\HH}{\mathbb{H}}
\dmo{\Isom}{Isom}
\dmo{\MCG}{MCG}
\dmo{\MPL}{MPL}
\dmo{\Mod}{\mathcal{M}}
\dmo{\PL}{PL}
\nc{\Sphere}{S}
\dmo{\sys}{sys}
\dmo{\Teich}{\mathcal{T}}
\nc{\Torus}{T}
\dmo{\vol}{vol}
\dmo{\WP}{WP}

\dmo{\convTV}{\;\stackrel{\mathrm{TV}}{\longrightarrow}\;}
\nc{\ExV}[2]{\mathbb{E}_{#1}\left[#2\right]}
\dmo{\EE}{\mathbb{E}}
\nc{\Pro}[2]{\mathbb{P}_{#1}\left[#2\right]}
\dmo{\PP}{\mathbb{P}}
\nc{\distTV}[2]{\mathrm{d}_{\rm TV}\left(#1,#2\right)}
\dmo{\UU}{\mathbb{U}}
\nc{\Var}[2]{\mathbb{V}\mathrm{ar}_{#1}\left[#2\right]}

\dmo{\alt}{\mathfrak{A}}
\dmo{\Aut}{Aut}
\dmo{\Fix}{Fix}
\dmo{\Hom}{Hom}
\dmo{\PSL}{PSL}
\dmo{\Rep}{Rep}
\dmo{\sym}{\mathfrak{S}}

\newcommand{\ccl}[1]{\overline{#1}^{\mathrm{cong}}}
\newcommand{\eps}{\varepsilon}
\newcommand{\FF}{\mathbb F}
\newcommand{\bs}{\backslash}
\newcommand{\GL}{\mathrm{GL}}
\newcommand{\OG}{\mathrm O}
\newcommand{\isom}{\mathrm{Isom}}
\newcommand{\id}{\mathrm{Id}}
\def\N{\mathbb{N}} 

\begin{document}

\begin{abstract}
  Every closed orientable surface $S$ has the following property: any two connected finite covers of \( S \) of the same degree are homeomorphic (as spaces). In this, paper we give a complete classification of compact $3$-manifolds with empty or toroidal boundary which have the above property. We also discuss related group-theoretic questions.
\end{abstract}

\maketitle

%%%%%%%%%%%%%%%%%%%%%%%%%%%%%%%%%%%%%%%%%%%%%%%%%
%		I N T R O D U C T I O N
%%%%%%%%%%%%%%%%%%%%%%%%%%%%%%%%%%%%%%%%%%%%%%%%%

\section{Introduction}
Given a finitely generated group $G$ and $n\in \N$ 
we denote by $s_n(G)$  the number of subgroups of index $n$ and we denote by
$e_n(G)$ the number of \emph{isomorphism types} of subgroups of index $n$.

The study of  $s_n(G)$ has attracted a lot of interest over the years. For example, the well-known asymptote for the number of index \( n \) subgroups in a free group $F_2$ of rank 2 is
\[s_n(F_2) \sim n \cdot (n!) \quad \text{ as }\quad n\to\infty, \]
(see for instance \cite{Dixon} and \cite[Chapter 2]{Lubotzky_Segal}). Much more recently, similar asymptotes were determined for surface groups by M\"uller and Schlage-Puchta \cite{MP1}, Fuchsian groups by Liebeck and Shalev \cite{Liebeck_Shalev} and non-uniform lattices in higher rank simple Lie groups by Lubotzky and Nikolov \cite{Lubotzky_Nikolov}. Finer questions ask for the number of subgroups of a given type. For instance, in \cite{MP2}, M\"uller and Schlage-Puchta determine the statistics of given isomorphism types in free products and in \cite{Lubotzky,Goldfeld_Lubotzky_Pyber} Lubotzky--Goldfeld--Pyber and Lubotzky--Nikolov count the number of congruence subgroups in arithmetic groups. For an introduction to the subject of subgroup growth, we refer to the monograph by Lubotzky and Segal \cite{Lubotzky_Segal}.

In this paper we are interested in the growth of  the sequence $\{e_n(G)\}_{n\in \N}$. It is apparent that its behavior can be strikingly different from $\{s_n(G)\}_{n\in \N}$. For example, we have just seen above that the sequence $\{s_n(F_2)\}_{n\in \N}$ grows super exponentially whereas it is clear that the sequence $\{e_n(F_2)\}_{n\in \N}$ is constantly equal to one. Similarly it is evident that for any surface group~$G$ the sequence $\{e_n(G)\}_{n\in \N}$ is constantly equal to one. These and free abelian groups are the only examples we know of (infinite, residually finite) groups with $e_n \equiv 1$. To give an example where $e_n$ exhibits nontrivial growth let us describe its behaviour for Fuchsian groups (namely, discrete subgroup of $\PSL_2(\RR)$) containing torsion elements: we observe in Proposition \ref{fuchsian_distinct} below that for a  Fuchsian group $G \subset \mathrm{PSL}_2(\mathbb R)$ of finite covolume, $e_n(G)$ grows polynomially. For example, we obtain (using M\"uller and Schlage--Puchta's asymptotic mentioned above) that 
\[
e_n(\mathrm{PSL}_2(\mathbb Z)) \sim \frac {n^2} 6 \quad \text{ as }\quad n\to\infty. 
\]
We note that this asymptotic contains geometric information about the orbifold $\mathrm{PSL}_2(\ZZ) \bs \mathbb H^2$: in general, the exponent ($2$ in this case) equals the number of nontrivial divisors of the orders of its cone points. Moreover, in the case of $\mathrm{PSL}_2(\mathbb Z)$, the leading coefficient ($1/6$) appears as the inverse of the product these divisors. 

To give our topic a more topological flavour, given a compact, connected manifold $M$ let us denote by $e_n(M)$ the number of homeomorphism types of connected degree $d$ covering spaces of $M$. All the examples above are essentially 1- or 2-dimensional. For a topologist, after the case of surfaces has been settled it is natural to next consider the case of 3-manifolds. A more pragmatic reason for this is that if $M$ is a closed, orientable, aspherical $3$-manifold, then we have $e_n(M)=e_n(\pi_1(M))$. This follows from the geometrization theorem, the rigidity theorem~\cite{Mos}, and work of Waldhausen~\cite{Waldhausen} and Scott~\cite{Sco2} (see e.g.~\cite[Theorem~2.1.2]{AFW}). Thus in this case the group-theoretical and topological problems are equivalent. Another is that the current understanding of 3--manifold groups (as opposed to higher-dimensional aspherical manifolds) is good enough to let us actually prove some nontrivial results. 

For technical reasons it is often more convenient to study the following closely related sequence
\[
e_n'(M)\,:=\, \sup \{ e_k(M)\,|\, k\leq n\}.
\]
In Section~\ref{quant} we point out that it follows from  \cite{Ago}, \cite{Cooper_Long_Reid_VH}, \cite{BGLM} and  \cite[Section 5.2]{BGLS}  that $e_n'(M)$ grows as a power of a factorial  for hyperbolic 3-manifolds. On the other hand it is clear that $e_n'(M)$ is a constant sequence for some choices of $M$, e.g.\ for the 3-torus and for 3-manifolds with cyclic fundamental group. As hyperbolic manifolds are in many ways generic among 3--manifolds this leads us to the following definition. 

\begin{definition}
  A manifold $M$ is called \emph{exceptional} if $e_n'(M)=1$ for all $n\in \N$. In other words, $M$ is exceptional, if for all $n\in \N$ any two connected degree $n$ covers of $M$ are homeomorphic. 
\end{definition}

Note that there is an analogous notion of exceptional groups, that is, a group is said to be \emph{exceptional} if any two index $n$ subgroups are isomorphic. 

As a warm-up, in  Lemma~\ref{lem:exceptional-surfaces} we determine which compact 2-dimensional manifolds are exceptional (this makes for a surprisingly amusing exercise). In our main theorem we determine all exceptional  compact $3$-manifolds with empty or toroidal boundary. 

\begin{maintheorem}\emph{ Let $M$ be a compact $3$-manifold with empty or toroidal boundary. Then $M$ is exceptional if and only if it is homeomorphic to one of the following manifolds:
\begin{enumerate}[font=\normalfont]
\item $k \cdot S^1 \times S^2$ for $k\geq1$,
\item $S^1\widetilde{\times} S^2$,
\item $S^1\times D^2$, 
\item $T^2\times I$,
\item $T^3$,
\item all spherical manifolds except those with fundamental group $P_{48}\times \ZZ/p$ with $\gcd(p,3)=1$ and $p$ odd, or $Q_{8n}\times \ZZ/q$ with $\gcd(q,n)=1$, $q$ odd, and $n\geq 2$. \label{mainthm_spherical}
\end{enumerate}}
\end{maintheorem}

In the theorem above and henceforth, $S^n$ denotes the $n$-sphere, $D^n$ denotes the $n$-dimensional disk, $T^n$ denotes the $n$-torus, $I$ denotes the unit interval $[0,1]$, and $k\cdot M$ denotes the $k$-fold connected sum of the manifold $M$. The groups mentioned in item (\ref{mainthm_spherical}) are defined in Section \ref{sec_spherical}.

Some of the techniques we use in the proof of the main theorem apply in a a larger setting. For instance, as noted above, it follows from work of Agol et al that hyperbolic 3-manifolds are not exceptional. We also give a different proof of this fact that generalizes to  lattices in most semisimple Lie groups as follows. (Recall that two Lie groups are called \emph{locally isomorphic} if they have isomorphic Lie algebras.)

\begin{prprep}{\ref{lattices}}
  Let $\Gamma$ be an irreducible lattice in a semisimple  linear Lie group not locally isomorphic to $\mathrm{PGL}_2(\RR)$. Then $\Gamma$ is not exceptional. 
\end{prprep} 

Regarding $\mathrm{PGL}_2(\RR)$ we noted (see above and Proposition \ref{fuchsian_distinct}) that Fuchsian groups with torsion exhibit nontrivial growth for $e_n$ and in particular are not exceptional. The latter is also true of Euclidean 2-orbifolds, and to prove this together with the fact that $T^3$ is the only exceptional Euclidean $3$-manifold we use arguments which lead to the following generalization. 

\begin{prprep}{\ref{proposition_euccase}}
  Let \( E \) be a Euclidean space and \( \Gamma \) a lattice in \( \isom(E) \). Then \( \Gamma \) is exceptional if and only if it is free abelian. 
\end{prprep}

Moreover for hyperbolic $3$-manifolds we can produce \textit{regular} non-homeomorphic covering spaces of equal degree:

\begin{prprep}{\ref{main_normal}}
  Let $\Gamma$ be the fundamental group of a complete hyperbolic $3$-manifold of finite volume. Then there exist sequences $c_n, d_n \to +\infty$ such that for each $n$ we can find at least $c_n$ normal subgroups of index $d_n$ in $\Gamma$, which are pairwise non-isomorphic.
\end{prprep}

\subsection{Questions}

There are many more questions about $e_n(M)$ that arise naturally.

\subsubsection{Quantitative questions}

The only question we deal with in general is whether or not $e'_n(M)=1$ for all $n\in \NN$. Once it is known that this is \emph{not} the case, it would of course be interesting to know how $e_n(M)$ and $e_n'(M)$ behave as functions of $n$.

For any hyperbolic $3$-manifold $M$, as we already noted above, there exists a constant $\alpha>0$ so that
\[e_{n}'(M) \geq (n!)^{\alpha},\] 
for all large enough $n\in \NN$. However, as of yet, there is no control on $\alpha$ in terms of, for instance, the hyperbolic geometry of $M$. We note that an upper bound on the smallest index of a subgroup of $\pi_1(M)$ that surjects onto $F_2$ would give a lower bound on $\alpha$.

On the other hand, our constructions of non-homeomorphic covers for Seifert fibered manifolds produce (in general) much sparser sequences of covers and we do not have an expectation for the answer to the quantitative question above.
 
The following is a sample of concrete questions we would like to know the answers to. 

 \begin{enumerate}
 \item Which growth types (polynomial, intermediate, exponential, factorial) appear for $e_n'(M)$?
 \item What are the 3-manifolds for which $e_n'(M)$ is polynomial, intermediate, exponential, factorial?
 \item Are there hyperbolic 3-manifolds for which $e_n(M)$ is not monotone even for large $n$?
 \item Given a hyperbolic 3-manifold do $e_n(M)$ and $e_n'(M)$ have the same growth type?
 \item If $M$ is a hyperbolic manifold, what information on the geometry of $M$ is encoded in the sequence $\{e_n(M)\}_{n\in\NN}$?
 \end{enumerate}

 We note that for general $3$-manifolds, many of the question above are also still open for the sequence $\{s_n(\pi_1(M))\}_{n\in\NN}$. For $2$-dimensional manifolds we know more. For instance, it is known that for a hyperbolic $2$-orbifold $M$, the sequence $\{s_n(\pi_1(M))\}_{n\in\NN}$ does encode some geometric information: the factorial growth rate of $s_n(\pi_1(M))$ is linear in the area of the orbifold \cite{MP1,Liebeck_Shalev}. On the other hand, for hyperbolic $3$-manifolds such a simple relation could never hold: there are hyperbolic $3$-manifolds $M$ of arbitrarily large volume but with $\mathrm{rank}(\pi_1(M))\leq 5$. The latter implies that the factorial growth rate of the sequence $\{s_n(\pi_1(M))\}_{n\in\NN}$ cannot be more than $4$.

\subsubsection{Stronger versions of non-exceptionality}

In most cases where we establish that a group $G$ is not exceptional, we do so by providing, for infinitely many distinct values of $d$, pairs $(G_1, G_2)$ of non-isomorphic subgroups of $G$ with index $d$ in $G$. We are not aware of an example of an infinite residually finite group which is not exceptional, but for which this stronger property is in default.

\subsubsection{Non-exceptionality for other classes of groups}

One may inquire about the exceptionality, or lack thereof, of other interesting classes of groups. Here are some examples:
  \begin{enumerate}  
  
  \item The only exceptional right-angled Artin groups are the free groups and the free abelian groups\footnote{This can be proved using the fact that all other RAAGs surject onto $\ZZ^2\ast \ZZ$.}. What about more general Artin groups? 

  \item Which Coxeter groups are exceptional? It seems reasonable to expect that only a few finite ones are.

  \item Are non-abelian polycyclic groups always non-exceptional?

 \item Are there other examples of non-elementary word-hyperbolic groups, besides surface groups and free groups, that are exceptional?
  \end{enumerate}

Note that all the groups in the first three items are linear and finitely generated and hence, by Malcev's theorem \cite{Malcev}, residually finite. As such, they at least have infinitely many different finite index subgroups (given that they are infinite themselves).

\subsection{Structure of the proof}

Our proof consists of a case by case analysis. That is, we use the prime decomposition theorem and Perelman's work to divide up $3$-manifolds into various geometric classes that we then address separately.

For most of the paper, we consider prime, orientable $3$-manifolds. Other than certain simple cases (Proposition~\ref{prop:specialones}), we show that prime, compact, orientable $3$-manifolds with non-empty toroidal boundary are not exceptional in Proposition~\ref{prp_JSJ}. Among the closed, prime, orientable $3$-manifolds, we see that $S^1\times S^2$ is clearly exceptional (Proposition~\ref{prop:specialones}), so it only remains to consider closed, irreducible, orientable $3$-manifolds. Those with a trivial JSJ decomposition are divided into the following classes:
\begin{itemize}
\item hyperbolic $3$-manifolds (Proposition \ref{proposition_hypcase}), 
\item Euclidean $3$-manifolds (Proposition \ref{prop:euclidean}), 
\item spherical $3$-manifolds (Proposition \ref{prp_spherical}), 
\item and the remaining Seifert fibered $3$-manifolds (Proposition \ref{prp_seif}).
\end{itemize}

\begin{figure}[h]
\begin{tikzpicture}[node distance=10mm and 5mm]
  \tikzstyle{box} = [rectangle, draw, text width=10em, text centered, rounded corners, minimum height=3em]
  \tikzstyle{endbox} = [rectangle, draw, text width=10em, text centered, rounded corners, minimum height=3em, fill=gray!30]

  \node (root) [box, very thick] {$M^3$ compact with empty or toroidal boundary} ;
  \node (prime) [box,below=of root] {Prime} ; 
  \node (nonprime) [box,right=of root] {Non-prime} ;
  \node (end-nonprime) [endbox,right=of nonprime] {Exceptional iff \\$M\cong k\cdot S^1 \times  S^2$, $k>1$ (Proposition~ \ref{prop:nonprime})} ;
	\draw[->]
  (root) edge (prime) (root) edge (nonprime) (nonprime) edge (end-nonprime);
  
  \node (orientable) [box,below=of prime] {Orientable} ; 
  \node (nonorientable) [box,right=of prime] {Non-orientable} ;
  \node (end-nonorientable) [endbox,right=of nonorientable] {Exceptional iff $M\cong S^1 \widetilde{\times} S^2$ (Proposition~\ref{prop:non-orientable-prime})} ;
  \draw[->]
  (prime) edge (orientable) (prime) edge (nonorientable) (nonorientable) edge (end-nonorientable);
  
  \node (irreducible) [box,below=of orientable] {Irreducible} ; 
  \node (S1xS2) [endbox, right=of orientable] {$M\cong S^1 \times  S^2$ \\$\Rightarrow$  exceptional (Proposition~\ref{prop:specialones})};
  \draw[->]
  (orientable) edge (irreducible) (orientable) edge (S1xS2);
  
  \node (nonclosed) [box, right=of irreducible] {Non-empty boundary} ;
  \node (specialboundary) [endbox, right=of nonclosed] {$M\cong T^2\times I$ or $S^1\times D^2$\\$\Rightarrow$ exceptional\\ (Proposition~\ref{prop:specialones})} ; 
  \node (generalnonclosed) [endbox,below right=of nonclosed] {Other cases \\$\Rightarrow$ not exceptional\\ (Propositions~\ref{prop:specialones},~\ref{prp_JSJ})} ;
  \node (closed) [box,below=of irreducible] {Closed} ; 
  \draw[->]
  (irreducible) edge (closed) (irreducible) edge (nonclosed) (nonclosed) edge (specialboundary)  (nonclosed) edge (generalnonclosed);
  
  \node (trivialJSJ) [box,below=of closed] {Trivial JSJ \\ decomposition} ; 
  \node (nontrivialJSJ) [box,below right=of closed] {Non-trivial JSJ decomposition} ;
  \node (sol) [endbox,right=of nontrivialJSJ] {Sol manifold \\$\Rightarrow$ not exceptional (Proposition~\ref{prp_sol})} ;
  \node (nonsol) [endbox,below right=of nontrivialJSJ] {Not Sol\\ $\Rightarrow$ not exceptional (Proposition~\ref{prp_JSJ})} ;
  \draw[->]
  (closed) edge (trivialJSJ) (closed) edge (nontrivialJSJ) (nontrivialJSJ) edge (sol) (nontrivialJSJ) edge (nonsol) ;
  
  \node (hyperbolic) [endbox,below =of trivialJSJ] {Hyperbolic \\ $\Rightarrow$ not exceptional (Proposition~\ref{proposition_hypcase})} ;
  \node (seifertfibered) [box,below right=of trivialJSJ] {Seifert fibered} ; 
  \draw[->]
  (trivialJSJ) edge (hyperbolic) (trivialJSJ) edge (seifertfibered);

  \node (spherical) [box,below left=of seifertfibered] {Covered by $S^3$} ;   
  \node (euclidean) [box,below=of seifertfibered] {Covered by $T^3$} ;
  \node (generic-seifert) [box, below right=of seifertfibered] {Other cases} ;
  \draw[->]
  (seifertfibered) edge (euclidean) (seifertfibered) edge (spherical) (seifertfibered) edge (generic-seifert);
  
  \node (end-sphere) [endbox,below=of spherical] {See Proposition~\ref{prp_spherical}};
  \node (end-gen-seif) [endbox,below=of generic-seifert] {Not exceptional (Proposition~\ref{prop:seifert-fibered-for-leitfaden})};
  \node (end-euc) [endbox,below=of euclidean] {Exceptional iff \\ $M\cong T^3$\\ (Proposition~ \ref{prop:euclidean})};]
  \draw[->]
  (euclidean) edge (end-euc) (spherical) edge (end-sphere) (generic-seifert) edge (end-gen-seif);
  
%\draw[->, to path={-| (\tikztotarget)}]
%  (cover) edge (solid-torus) (cover) edge (generic-seifert);
\end{tikzpicture}
\caption{Leitfaden}
\label{fig_diagram}
\end{figure}
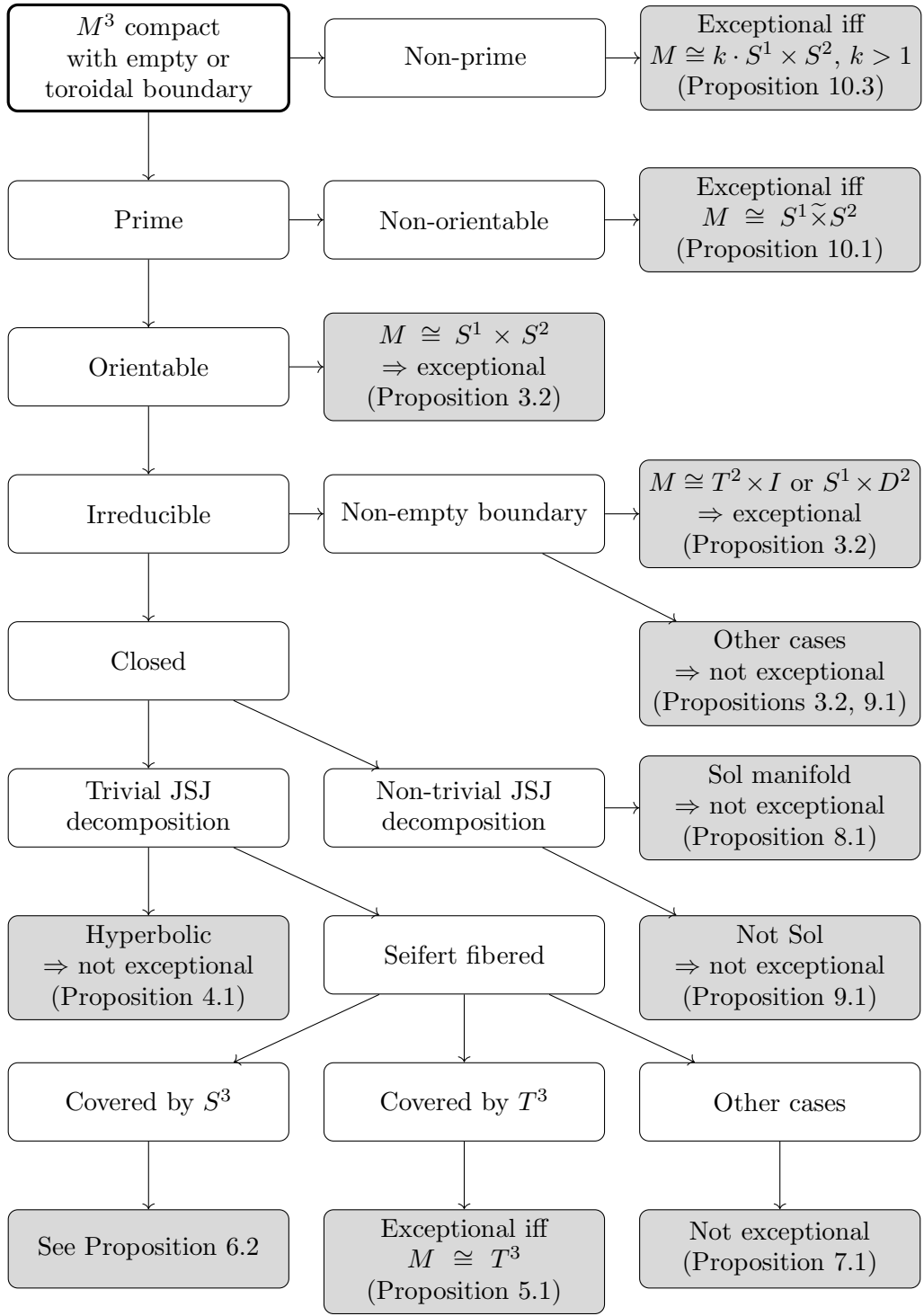

Then we treat closed, irreducible, orientable $3$-manifolds with a non-trivial JSJ decomposition, where we need two separate arguments - one for Sol manifolds (Proposition \ref{prp_sol}) and one for all others (Proposition \ref{prp_JSJ}). 
The classification of prime non-orientable exceptional $3$-manifolds and non-prime exceptional $3$-manifolds is an almost direct consequence of our work with prime and orientable $3$-manifolds and is given in Section~\ref{sec:nonprime}. The diagram in Figure \ref{fig_diagram} shows the structure of the proof.

\subsection{Conventions}
All manifolds are assumed to be compact and connected. We will call a compact manifold with non-empty toroidal boundary \emph{hyperbolic} if its interior admits a complete hyperbolic metric of finite volume. Usually we do not distinguish between a manifold and its homeomorphism type.

\subsection{Acknowledgements}
SF is grateful for the support provided by the SFB 1085 ``higher invariants''  funded by the DFG. JR was supported by ANR grant ANR-16-CE40-0022-01-AGIRA. 
Most of the work on the paper was done while various subsets of the authors met at ICMAT Madrid, the Max Planck Institute for Mathematics in Bonn, the University of Bonn, and the University of Regensburg. We are very grateful for the hospitality of these institutions.  Lastly, we thank the anonymous referee for the careful reading of the paper and for the useful comments and suggestions.

%%%%%%%%%%%%%%%%%%%%%%%%%%%%%%%%%%%%%%%%%%%%%%%%%
%		B A C K G R O U N D   F A C T S
%%%%%%%%%%%%%%%%%%%%%%%%%%%%%%%%%%%%%%%%%%%%%%%%%

\section{Surfaces and two-dimensional orbifolds}

\subsection{Surfaces}

As a warm-up we state the following fairly elementary lemma, which doubles as a great exam problem in a first course on topology.

\begin{lemma}\label{lem:exceptional-surfaces}
The only exceptional compact 2-dimensional manifolds are the disk $D^2$, the annulus, the M\"obius band, the real projective plane $\RR \textup{P}^2$, and all closed orientable surfaces.
\end{lemma}

\begin{proof}
It is clear that the surfaces listed are exceptional. (For closed orientable surfaces this is an immediate consequence of the classification of closed orientable surfaces in terms of their Euler characteristic and the multiplicativity of Euler characteristic under finite covers.)

Next assume that $M$ is an orientable surface with at least one boundary component and that $M$ is neither a disk nor an annulus. After possibly going to a finite cover we can assume that $M$ has $k\geq 3$ boundary components. By giving the boundary components the orientation coming from~$M$, the boundary of $M$ induces a summand of $H_1(M)$, that is naturally isomorphic to
\[\left(\bigoplus_{i=1}^k \ZZ\; a_i \right) \bigg/ \ZZ\; (a_1+\cdots +a_k).\] 

Choose an epimorphism $\varphi\colon \pi_1(M)\to \Z/k$  such that $\varphi(a_1) = 1$, $\varphi(a_2)=-1$, and $\varphi(a_i)=0$ for $i\neq 1,2$. On the other hand, we can also find an epimorphism $\psi \colon \pi_1(M)\to \Z/k$ such $\psi(a_i)\ne 0$ for all $i$. But then the covers corresponding to $\ker(\varphi)$ and $\ker(\psi)$ have different numbers of boundary components, so they are not homeomorphic.

Now suppose that $M$ is a non-orientable surface that is not homeomorphic to either $\RR\textup{P}^2$ or the M\"obius band. There exists precisely one $2$-fold cover of $M$ that is orientable. On the other hand, we have $H^1(M;\Z/2)\cong (\Z/2)^k$ for some $k\geq 2$. Since $k\geq 2$, there exists at least one other $2$-fold cover. This shows that $M$ is not exceptional.
\end{proof}

%%%%%%%%%%%%%%%%%%%%%%%%%%%%%%

\subsection{Fuchsian groups}

If $\Gamma$ is a finitely generated Fuchsian group of finite covolume, then the quotient $O = \Gamma \bs \HH^2$ is an orbifold of finite type, that is, a compact surface with genus $g$ and $k$ punctures as well as conical points with angles $2\pi/m_i$, $1 \le i \le s$, for some $g,s,k$. We will call the tuple $(g, k, m_1, \ldots, m_s)$ the {\em signature} of $\Gamma$. 
%$(g, k, (m : k_m), m\ge 2)$ where $k_m$ is the number of singular points with angle $2\pi/m$ the {\em signature} of $\Gamma$ (this is slightly non-standard usage, usually the signature is rather $(g, k, m_1, \ldots, m_s)$).
The orbifold Euler characteristic of $O$ is defined by 
\[
-\chi(O) := 2g - 2 + k + \sum_{i=1}^s \left( 1 - \frac 1 {m_i} \right).
\]
It is multiplicative in covers, that is, if $\Gamma'$ is a finite index subgroup of $\Gamma$ and $O' = \Gamma' \bs \HH^2$ then $\chi(O') = [\Gamma : \Gamma'] \cdot \chi(O)$. 

\begin{proposition} \label{fuchsian_distinct}
  Let $\Gamma \subset \mathrm{PSL}_2(\mathbb R)$ be a Fuchsian group of finite covolume with signature $(g, k, m_1, \ldots, m_s)$. Let $d$ be the number of distinct divisors of all $m_i$. Then 
  \begin{equation} \label{bds_fuchsian}
  n^{d-1} \ll e_n'(\Gamma) \ll n^{d-1}.
  \end{equation}
\end{proposition}

In specific cases one can say more by studying the proof of the upper and lower bounds more carefully; for example $e_n(\PSL_2(\Z)) = n^2/6 + O(n)$ as we stated in the introduction: to obtain this we note that $e_n(\PSL_2(\Z))$ is equal to the number of solutions to the equation \eqref{euler_char_OK} below by \cite[Theorem B]{MP2}. Computing the covolume of the lattice appearing in the proof of the upper bound below (an exercise in Euclidean geometry) we get the result. We could have used M\"uller and Schlage-Puchta's result to get a more precise asymptote in more cases (though the computation of constants would require further analysis), but as it does not have the generality we require (in particular, it is not clear whether an analogous result holds for cocompact Fuchsian groups) we elected to give a more direct and elementary proof which works for all Fuchsian groups. 

\begin{proof}
  For the upper bound in \eqref{bds_fuchsian} we note that two Fuchsian groups with respective signatures $(g, k, m_1, \ldots, m_s)$ and $(h, l, n_1, \ldots, n_r)$ are isomorphic if and only if $(m_1, \ldots, m_s) = (n_1, \ldots, n_r)$ and either $k = 0 = l$ and $g = h$ or $k, l > 0$ and $2g+k = 2h+l$. In addition, the condition that $k = 0$ is invariant under taking finite index subgroups. %are isomorphic if and only if they have the same signature;
  
  Thus bounding the number of possible tuples $(g+k, m_1, \ldots, m_s)$ for a degree $n$ cover of $O:=\Gamma \bs \HH^2$ with signature $(g, k, m_1, \ldots, m_s)$ gives an upper bound for $e_n(\Gamma)$. Let $\Omega$ be the set of divisors of elements of the set $\{m_i\}$. Let $d$ denote the cardinality $d := |\Omega|$. If $\Gamma'$ is a finite index subgroup of $\Gamma$ with signature $(g', k', m_1', \ldots, m_t')$ we further set $r := 2g' - 2 + k'$, and 
  \[
  k_m = |\{ 1\le i \le t : m_i' = m\}| 
  \]
  for $m \in \Omega$, so that (as each $m_i'$ must divide some $m_i$) we have $\sum_{i=1}^t \left( 1 - \frac 1 {m_i'} \right) = \sum_{m\in\Omega,\, m\not=1} \left( 1 - \frac 1 {m} \right)k_m$.
  We see by multiplicativity of Euler characteristic that, letting $n = |\Gamma/\Gamma'|$ be the index, we have: 
  \begin{equation} \label{euler_char_OK}
  r + \sum_{\substack{m \in \Omega \\ m \not= 1}} \left(1 - \frac 1 m \right)k_m = n\chi(O). 
  \end{equation}
  and hence the number of non-isomorphic subgroups of $\Gamma$ is at most the number of points $(r/n, k_m/n : m\in \Omega, m\not=1)$ of the lattice $\frac 1 n \Z^d$ which belong to the $(d-1)$-simplex given by the intersection of the hyperplane $x_0 + \sum_{m \in \Omega,\, m\not=1} (1 - 1/m)x_m = \chi(O)$ with the positive quadrant. Now if $P$ is a polygon in an affine space $E$ and $L$ a lattice in $E$ we have $|L \cap nP| \ll n^{\dim E}$, and the upper bound follows\footnote{More precisely $|L \cap nP| \sim Cn^{\dim P}$ where $C$ is the quotient of the volume of $P$ by the covolume of $L$, so computing these volumes gives an explicit asymptotic upper bound. }. 

  \medskip

  To prove the lower bound in \eqref{bds_fuchsian} we construct explicit subgroups which contain distinct numbers of conjugacy classes of elements of order $m_i$. To do this we use the following presentation for the Fuchsian group $\Gamma$ of signature $(g, k, m_1, \ldots, m_s)$:
  \begin{equation} \label{pres_fuchsian}
  \Gamma = \left\langle x_1, \ldots, x_s, p_1, \ldots, p_k, a_1, b_1, \ldots, a_g, b_g \,\middle|\, x_i^{m_i},\,  \prod_i x_i \prod_j p_j \prod_l [a_l, b_l] \right\rangle.  
  \end{equation}
  
  Our first step is to reformulate our problem in terms of morphisms $\rho \colon \Gamma \to \sym_n$, where $\sym_n$ denotes the symmetric group on $n$ letters. This will use the well known correspondence between subgroups of index $n$ and transitive permutation representations  $\rho \colon \Gamma \to \sym_n$ (see for instance \cite{Lubotzky_Segal}). We start with the following lemma. Given a group $G$, $Z(x)$ will denote the centraliser of $x \in G$ and $x^G$ will denote the conjugacy class of $x$ in $G$. 

  \begin{lemma} \label{fixed_points_sing}
   Let $G$ be a group and $\rho \colon G \to \sym_n$ a homomorphism with transitive image; let 
   \[
   H =\mathrm{Stab}_\rho \{1\} < G
   \] 
   be the stabiliser of $1$ in $G$ via $\rho$. Then the number of conjugacy classes in $H$ containing an element of $x^G$ is equal to the cardinality of the quotient set %orbit set 
   \[
   \raisebox{-5pt}{$\rho(Z(x))$} \mspace{-6mu} \diagdown \mspace{-6mu} \raisebox{5pt}{$\st{ a\in \{1,\ldots,n\} }{ \rho(x)\cdot a = a}$}. %\rho(Z(x)) \;\bs \; \st{ a\in \{1,\ldots,n\} }{ \rho(x)\cdot a = a} .
   \]
  \end{lemma}

  \begin{proof}
    Write $Z=Z(x)$. The orbits of the right-conjugation action of $H$ on $x^G$ are in bijection with $Z \bs G / H$. The orbit map $G/H \to \{1, \ldots, n\}$ is bijective (using transitivity) and $Z$-equivariant so $Z \bs G / H \leftrightarrow \rho(Z) \bs \{1, \ldots, n\}$. Finally, the last map (bijectively) sends conjugacy classes of $x$ lying in $H$ to $\rho(Z)$-orbits fixed points of $\rho(x)$.   \end{proof}
    
  We will distinguish between non-isomorphic finite index subgroups of the same index by comparing the numbers of conjugacy classes of maximal finite cyclic subgroups (on the orbifold side this corresponds to the cone points together with their orders). We will first deal with the cases where the algebraic structure of $\Gamma$ is simpler, that is when we are in one of the following cases : $k \ge 2$ ($O$ has at least 2 cusps), $k \ge 1, g \ge 1$ (at least 1 cusp and genus at least 1) or $g \ge 2$. We will deal with the remaining cases with a bootstrap argument. 

  Let $m \in \Omega$, $m \not= 1$. Let $\rho$ be a morphism $\Gamma \to \sym_n$ and $H = \mathrm{Stab}_\rho(1)$ as in the statement of the lemma above. Then if $x \in H$ is a primitive element\footnote{Recall that a group element is primitive if it generates the maximal cyclic subgroups it lies in.} of order $m$, we have that $m$ is a divisor of some $m_i$ and $x=gx_i^{pm_i/m}g^{-1} \in H$ for some $p \in \ZZ$ coprime to $m$, but $gx_i^{m_i/q}g^{-1} \not\in H$ for any $m\mid q \mid m_i$, $m < q$. If we fix such $i$ the corresponding conjugacy classes of primitive order $m$ elements correspond, via the lemma, to $\rho(Z(x_i))$-orbits on the set of $m/m_i$-cycles of $\rho(x_i)$ times the set of invertible elements of $\ZZ/m\ZZ$. When we consider only the maximal cyclic subgroups they generate we kill the $(\ZZ/m\ZZ)^\times$-factor; in conclusion, the number of conjugacy classes of maximal cyclic subgroups of $H$, which are of order $m$, is equal to the number of $\rho(Z(x_i))$-orbits on the set of $m/m_i$-cycles of $\rho(x_i)$. 

  Now in a Fuchsian group we have $Z(x) = \langle x \rangle$ for any nontrivial element. In particular, $\rho(Z(x))$ acts trivially on the cycles of $\rho(x)$. So the number $N$ of conjugacy classes of maximal cyclic subgroups of order $m$ in $H$ is given by the following formula : 
  \begin{equation} \label{formula_nb_cc}
  N = \sum_{i:\; m|m_i} \#\{ (m_i/m) \text{-cycles in the image of } \rho(x_i) \}.
  \end{equation}
  %This uses the fact that $Z(x_i) = \langle x_i\rangle $, so that in particular, the $\rho(Z(x_i))$-orbit of a fixed point of $\rho(x_i)$ has one element.
 
  Let $\Omega_i$ be the set of divisors of $m_i$, and choose subsets $1 \in \Omega_i' \subset \Omega_i$ so that $\Omega \setminus \{1\}$ is the disjoint union of all $\Omega_i' \setminus \{1\}$ ; in particular we have
  \begin{equation} \label{cardinalities}
    d - 1 = \sum_{i=1}^s \left(|\Omega_i'| - 1\right)
  \end{equation}
  Let $n \ge 2$. Consider collections $(\pi_1, \ldots, \pi_s)$, where for all $i$, $\pi_i$ is a partition of $n$ whose parts lie in $\{m_i/m : m \in \Omega_i'\}$. We claim (and will prove in the next paragraph) that for a positive proportion of such partitions it is possible to construct a representation of $\Gamma$ in the symmetric group $\sym_n$ such that the cycle decomposition of $\rho(x_i)$ corresponds to $\pi_i$. Moreover, for different partition-tuples the subgroups cannot be conjugated to each other: if $m \in \Omega, m\not= 1$ then maximal cyclic subgroups of order $m$ can only come from lifts of the only $x_i$ such that $m \in \Omega_i'$. It follows from \eqref{formula_nb_cc} that the number of such subgroups is equal to the $(m/m_i)$-part of $\pi_i$, and these numbers determine the $\pi_i$. 

  In this paragraph we prove the claim (under the hypotheses on $g, k$ stated above). Assume first that $k\ge 2$ or $k \ge 1$ and $g \ge 1$. Then we see from \eqref{pres_fuchsian} that $\Gamma$ decomposes as the free product
  \[
  \Gamma = \langle e_1 \rangle \ast \cdots \ast \langle e_s \rangle \ast F
  \]
  where $F$ is a free group. Then any collection is realisable, as we can pick any image we want for each $e_i$ and then complete it by a transitive representation of $F$. In the remaining cases where $k = 0$ and $g \ge 2$ then we can do the same thing, choosing $\langle \rho(a_1), \rho(b_1) \rangle$ to have transitive image and so that the product $\sigma = \prod \rho(e_i) [\rho(a_1), \rho(b_1)]$ is alternating, and finally $a_2, b_2$ so that the relation in \eqref{pres_fuchsian} holds (this is possible because every alternating permutation can be written as a commutator). %In general we can take a finite index subgroup in $\Gamma$ satisfying either $g > 0$ or $k >0$ and still containing torsion elements with the same order as those in $\Gamma$, so we can construct such representations for all multiples of the index and this is sufficient to obtain a positive proportion of all partitions. 

  For a finite subset $A \subset \N$ denote by $p_A(n)$ the number of partitions of $n$ into elements of $A$. The representations constructed above give rise to $\prod p_{\Omega_i'}(n)$ isomorphy classes of index $n$ subgroups of $\Gamma$.
  %such that for any pair there is a $m \in \Omega$ such that they do not have the same number of conjugacy classes of maximal cyclic subgroups of order $m$. In particular they are pairwise non-isomorphic.
  Thus we have proven that for any $n$ we have: 
  \[
  e_n(\Gamma) \ge \prod_{i=1}^s p_{\Omega_i'}(n). 
  \]
  It is an easily observed fact that $p_A(n) \gg n^{|A|-1}$ (see also \cite[pp. 458--464]{Nathanson} for a more precise asymptotic), so that we get from the previous inequality that
  \[
  e_n(\Gamma) \gg n^{\sum_{i=1}^s (|\Omega_i'| - 1)}
  \]
  and using \eqref{cardinalities} we can conclude that $e_n'(\Gamma) \gg n^{d-1}$. 

  \medskip

  To finish the proof in the remaining cases (that is, those that do not satisfy our assumptions on $k$ and $g$), we pass to a finite cover $O'$ which satisfies the hypotheses we used above and contains torsion elements of (at least) the same orders as $O$ (see next lemma). Let $f$ be the degree of this cover, then we get from what we previously proved that $e_{fd}(\Gamma) \gg f^{d-1}$ for all $n$, which immediately implies that $e_n'(\Gamma) \gg n^{d-1}$. 
\end{proof}

The following fact that we used at the end of the proof might be standard (and using results of Liebeck--Shalev our proof is almost immediate) but we did not find a reference for it.

\begin{lemma}
  Let $\Gamma$ be a Fuchsian group. Then it has a finite index subgroup $H$ whose signature $(g, k, \ldots)$ satisfies either $k \ge 2$ or $g \ge 2$ and such that for any $m$ if $\Gamma$ contains an element of order $m$ the same goes for $H$. 
\end{lemma}

\begin{proof}
%{\color{red} Alternative proof:} 
%This proof is based on a counting argument: results of Liebeck that the number of index $n$ subgroups in which the required torsion elements are present and contain boundedly many conjugacy classes (so the genus or number of cusps has to be . In fact, we will see that the total number of subgroups with the right torsion grows factorially fast, whereas the number of subgroups of bounded genus grows only exponentially. From this it then follows that there must be at least one finite index subgroups of which the genus is more than $2$ and which still contains torsion elements of the same orders as in $\Gamma$. The growth of the total number of subgroups in which torsion elements survive follows directly from Liebeck -- Shalev.
%We will again use the connection between elements of $\Hom(\Gamma,\sym_n)$ and index $n$ subgroups. For ease of notation, let us assume $\Gamma$ contains only orientation preserving isometries. %and is cocompact (the general case readily follows from this).
%\[\Gamma = \left< a_1,b_1,\ldots,a_g,b_g, x_1,\ldots,x_s,  \middle|\; \begin{array}{c} x_1\cdots x_s\cdot \cdot [a_1,b_1]\cdots [a_g,b_g], \\[2mm] x_1^{m_1},\ldots,x_s^{m_s} \end{array} \right> \]Here the numbers $m_1,\ldots m_s$ are exactly the orders of torsion that appear in $\Gamma$.
Recall that the group $\Gamma$ admits a presentation of the form \eqref{pres_fuchsian}. In \cite{Liebeck_Shalev} Liebeck -- Shalev prove the following: let $C_i$ be the conjugacy class of an element of $\sym_n$ which has only $m_i$-cycles and $f_i$ fixed points, with $f_i$ bounded independently of $n$. Let $\mathbf C = (C_1, \ldots, C_s)$ and $\Hom_{\mathbf C}(\Gamma, \sym_n)$ the set of homomorphisms $\rho : \Gamma \to \sym_n$ such that $\rho(x_i) \in C_i$ (recall from \eqref{pres_fuchsian} that $x_i$ is an element of order $m_i$ in $\Gamma$). Then
\[
\left| \Hom_{\mathbf C}(\Gamma, \sym_n) \right| \sim (n!)^{1-\chi(\Gamma)} \cdot n^{O(1)}
\]
(loc. cit., Theorem 3.5) and an element in $\Hom_{\mathbf C}(\Gamma, \sym_n)$ is transitive with probability tending to 1 as $n \to +\infty$ (loc. cit., Theorem 4.4).

We choose all $f_i$ such that $1 \le f_i \le m_i$ (there is a unique possible choice for each $n$ and $i$). Let $\rho \in \Hom_{\mathbf C}(\Gamma, \sym_n)$ be transitive (such $\rho$ exists by the results presented above) and let $H$ be the subgroup $\mathrm{Stab}_\rho\{1\}$ and $O'$ the corresponding orbifold cover of $O$, with genus $g'$ and $k'$ cusps. Then we have that
\[
2g'+ k' = n \chi(O) - \sum_{i=1}^s \left( 1 - \frac 1{m_i} \right) f_i + 2 =  n \chi(O) - O(1)
\]
and the right-hand side goes to infinity, so taking $n$ large enough we can ensure $2g' + k \ge 4$, which implies either $k' \ge 2$ or $g' \ge 2$. Moreover, since $f_i \ge 1$ if follows from Lemma \ref{fixed_points_sing} that the group $H$ contains an element of order $m_i$ for each $i$. So this $H$ satisfies the conclusion of the lemma. 
\end{proof}
\section{Preliminaries}

Let us start our discussion of $3$-manifolds with some preliminary observations. Recall that a  group $G$ is called \emph{residually finite} if 
\[\bigcap_{\substack{H \triangleleft G \\ [G:H] < \infty}} H = \{e\},\]
where $e\in G$ denotes the unit element. It follows from work of Hempel~\cite{Hem}, together with the proof of the geometrization theorem, that  the fundamental group of a  $3$-manifold has this property:

\begin{theorem}\label{theorem_resfin} \cite{Hem} Let $M$ be a compact $3$-manifold, then $\pi_1(M)$ is residually finite.
\end{theorem}
\noindent Next, we make some elementary observations about certain simple $3$-manifolds.
\begin{proposition}\label{prop:specialones}
The $3$-manifolds $S^3$, $T^3$, $T^2\times I$, $S^1\times D^2$, $S^1\times S^2$, and $S^1\widetilde{\times} S^2$ are exceptional. The twisted $I$-bundle over the Klein bottle is not exceptional.
\end{proposition}

\begin{proof}
It is an elementary exercise to verify that the manifolds mentioned in the first sentence are exceptional. 
For example, note that any manifold with cyclic fundamental group is exceptional.
Finally note that the twisted $I$-bundle over the Klein bottle has two $2$-fold covers, one of which is again homeomorphic to the twisted $I$-bundle over the Klein bottle and the other is homeomorphic to $T^2\times I$. Thus it is not exceptional.
\end{proof}

We conclude the section with the following elementary observation, which uses the fact that our covers need not be regular.

\begin{lemma}\label{lem:nonexceptional-cover}
If a manifold $M$ has a finite-sheeted cover $p\colon \widehat{M}\rightarrow M$ such that $\widehat{M}$ is not exceptional, then $M$ is not exceptional.
\end{lemma}

%%%%%%%%%%%%%%%%%%%%%%%%%%%%%%%%%%%%%%%%%%%%%%%%%
%		T H E   H Y P E R B O L I C   C A S E
%%%%%%%%%%%%%%%%%%%%%%%%%%%%%%%%%%%%%%%%%%%%%%%%%
\section{The hyperbolic case}

\noindent In this section, we prove the following result.

\begin{proposition} \label{proposition_hypcase}
  Hyperbolic $3$-manifolds of finite volume are not exceptional. 
\end{proposition}

As mentioned in the introduction, this follows by combining largeness (\cite{Ago}) with either \cite{Zimmermann} or \cite{BGLM,BGLS}, which yields a much stronger quantitative result. We also provide an independent proof which does not use largeness, and works in the more general setting of irreducible lattices in almost all semisimple Lie groups (Proposition \ref{lattices}). Finally we show that in any hyperbolic $3$-manifold group one can also find non-isomorphic {\em normal} subgroups with the same index (Proposition~\ref{main_normal}). 

%%%%%%%%%%%%%%%%%%%%%%%%%%%%%%%%%%%%%%%%%%%%%%%%%%%%%%%%%%%%

\subsection{Non-exceptionality of lattices in Lie groups}

Let $G$ be a semisimple Lie group and let $X$ be the symmetric space associated to $G$ (for example, $G = \mathrm{PGL}_2(\CC)$ and $X = \HH^3$). Then for any discrete subgroup $\Gamma \le G$, the quotient $\Gamma \bs X$ is a complete Riemannian orbifold locally isometric to quotients of $X$ by finite subgroups of $G$ (in particular, if $\Gamma$ is torsion-free then $\Gamma \bs X$ is a manifold). We will call such orbifolds {\em $X$-orbifolds}. 

The Mostow-Prasad rigidity theorem \cite{Mos, Mos2, Prasad_rigidity,Margulis_book} states that if $G$ is not locally isomorphic to $\mathrm{PGL}_2(\RR)$, then two irreducible lattices $\Gamma_1$ and $\Gamma_2$ in $G$ are isomorphic as abstract groups if and only if the orbifolds $\Gamma_i \bs X$ are isometric to each other. In particular the metric invariants of $\Gamma \bs X$ are an isomorphism invariant of $\Gamma$. We will be using the {\em systole} to distinguish between subgroups: given an $X$-orbifold $M$ this is defined as the infimum of lengths of closed geodesics on $M$, and we will denote it by $\mathrm{sys}(M)$. Note that it follows from the Margulis lemma that $\sys(M)$ is positive if $M$ has finite volume. 

The systole of $\Gamma \bs X$ can be computed from the action of $\Gamma$ on $X$. If $g \in G$ is an element whose semisimple part does not belong to a compact subgroup of $G$, then the minimal translation
\[
\ell(g) := \inf_{x \in X} d_X(x, gx)
\]
is positive. Then, denoting by $\Gamma_{ah}$ the set of such elements in $\Gamma$, we have: 
\begin{equation} \label{systole_tranlength}
  \sys(M) = \min_{\gamma\in\Gamma_{ah}} \ell(\gamma).
\end{equation}
Note that if $\Gamma$ is cocompact, then $\Gamma_{ah}$ is the set of semisimple elements of infinite order in $\Gamma$. 

\medskip

\noindent We will now prove the following result, of which Proposition \ref{proposition_hypcase} is a special case. 

\begin{proposition} \label{lattices}
  Let $\Gamma$ be an irreducible lattice in a semisimple linear Lie group not locally isomorphic to $\mathrm{PGL}_2(\RR)$. Then $\Gamma$ is not exceptional. 
\end{proposition}

\begin{proof}
  Let \( G \) be a semisimple linear Lie group as in the statement, so we may assume that $G < \GL_d(\mathbb R)$ for some $d$ and let \( \Gamma \) be a lattice in \( G \). It is a standard consequence of local rigidity of \( \Gamma \), which holds under the condition that \( G \) not be locally isomorphic to \( \mathrm{PGL}_2(\RR) \), that we may conjugate $G$ so that there exists a number field $F$ such that $\Gamma < \GL_d(F)$ (the proof given for \cite[Theorem 3.1.2]{MR_book} in the cocompact case adapts immediately to all other groups). Let $\mathbf H$ be the Zariski closure of $\Gamma$ in the  $\QQ$-algebraic group obtained by Weil restriction of the linear $F$-algebraic group \( \GL_d(F) \) to \( \QQ \). %\footnote{[SF] silly question: is $\GL_d/F$ the same as $\GL_d(F)$ or something else? [JR] it's a notation to indicate that we consider it as a $F$-algebraic group---in my opinion there's not much difference between the two so I put the clearer one}
  By passing to a finite index subgroup if necessary, we may assume that every finite index subgroup of $\Gamma$ has Zariski closure equal to $\mathbf H$. Indeed, every chain of finite index algebraic subgroups $\ldots < \Gamma_{i+1} < \Gamma_i < \ldots < \Gamma$ is necessarily finite. So, a chain of finite index subgroups so that the Zariski closures  are strictly contained in each other necessarily terminates after a finite number of steps and we may take the last term.

  By finite generation of $\Gamma$ there exists a finite set $S$ of rational primes such that $\Gamma \subset \mathbf H\left(\ZZ[p^{-1}, p \in S]\right)$. For the rest of the proof we will fix a rational prime $q \not\in S$. Thus we can define the group of $\ZZ_q$-points, $H_q = \mathbf H(\ZZ_q)$. Nori-Weisfeiler strong approximation \cite{Weisfeiler} implies that we can choose $q$ so that the closure of $\Gamma$ in $H_q$ is of finite index. Since $H_q$ is $q$-adic analytic we may assume that it is a uniform pro-$q$ subgroup (cf. \cite[Theorem 8.1]{DdSMS}), replacing $\Gamma$ by a finite index subgroup if necessary. 

  \medskip

\noindent Now we prove the following lemma.

  \begin{lemma} \label{kill_element}
    Let $p$ be a prime, $H$ a uniform pro-$p$ group, and $\gamma \in H$. There exists a sequence $(H_1(k), H_2(k))$ of pairs of open subgroups of $H$ such that $|H/H_1(k)| = |H/H_2(k)|$, $H_i(k+1) \subset H_i(k)$ and
    \[
    \bigcap_{k \ge 1} H_1(k) = \{e\}, \: \bigcap_{k \ge 1} H_2(k) = \overline{\langle\gamma\rangle}.
    \]
  \end{lemma}

  \begin{proof}
    Let $P_k(H)$ be the lower $p$-series of $H$ (see \cite[Definition 1.15]{DdSMS}). Replacing $H$ by some $P_k(H)$ we may assume that $\gamma \in H \setminus P_2(H)$. Uniformity of $H$ implies that, independently of $k \ge 1$, the group $P_k(H)/P_{k+1}(H)$ is an $\FF_p$-vector space of fixed dimension $c$ so that
    \[
    |H / P_{k+1}(H)| = p^{ck}.
    \]
    On the other hand, we have $\gamma^{p^k} \in P_{k+1}(H) \setminus P_{k+2}(H)$, so we get that
    \[
    \left| H / \left(\langle\gamma\rangle P_{k+1}(H) \right)\right| = p^{(c-1)k}.
    \]
    We define
    \[
    H_2(k) = \langle \gamma \rangle P_{c\cdot k+1}
    \]
    which satisfies that $H_2(k+1) \subset H_2(k)$ and $\bigcap_{k \ge 1} H_2(k) = \overline{\langle\gamma\rangle}$. Let also 
    \[
    H_1(k) = P_{(c-1)\cdot k+1}
    \]
    Then we have that
    \[
    |H/H_1(k)| = p^{(c-1)k} = p^{ck}/p^k = |H/H_2(k)|. 
    \]
    On the other hand, $P_{k-l_k+1}(H) \supset H_1(k)$ so that $\bigcap_{k \ge 1} H_1(k) = \{e\}$. 
  \end{proof}

  Now let $\gamma \in \Gamma$ be a semisimple element in $\Gamma$ of infinite order. Applying the lemma to $H_q$ and~$\gamma$ we get two sequences of subgroups
  \[
  \Gamma_1(q,k) = \Gamma \cap H_1(k+1), \: \Gamma_2(q,k) = \Gamma \cap H_2(k+1)
  \]
  which satisfy the same properties as $H_i$. In particular, for any finite set $\Sigma \subset \Gamma \setminus \{1\}$, we have $\Sigma \cap \Gamma_1(q,k) = \emptyset$ for large enough $k$. Applying this to the finite sets
  \[
  \Sigma_R = \{ \gamma \in \Gamma_{ah}:\: \ell(\gamma) \le R\}
  \]
  with $R$ going to infinity we see, using the formula \eqref{systole_tranlength}, that: 
  \[
  \lim_{k\to+\infty}\sys(\Gamma_1(q,k) \bs X) = +\infty.
  \]
  On the other hand, we have $\gamma \in \Gamma_2(k)$ for all $k > 0$ and it follows that
  \[
  \forall k > 0 : \: \sys(\Gamma_2(q,k) \bs X) \le \ell(\gamma)
  \]
  and in particular, for any large enough $k$, the systoles of the $X$-orbifolds $\Gamma_1(q,k) \bs X$ and $\Gamma_2(q,k) \bs X$ are different. It finally follows from Mostow's rigidity theorem, as observed before the proposition, that the subgroups $\Gamma_1(q,k)$ and $\Gamma_2(q,k)$, which have the same index, cannot be isomorphic to each other for large enough $k$. 
\end{proof}

%\subsubsection{Remark} \label{rem_Fuchsian} We can use essentially the same proof as above to prove that any cocompact Fuchsian group with torsion is exceptional: such groups can always be realised as subgroups of \( \mathrm{PGL}_2(F) \) for \( F \) a number field (this is easily seen: the representation variety is defined by equations with integer coefficients, hence the \( \overline\QQ \)-points are non-empty, and as the group is finitely generated such a representation takes values in a finitely generated field of \( \overline\QQ\), that is a number field). Then, using Lemma \ref{kill_element} we can construct subgroups with the same index, one of which contains a given non-trivial torsion element and the other is torsion-free. 

%%%%%%%%%%%%%%%%%%%%%%%%%%%%%%%%%%%%%%%%%%%%%%%%%%%%%%%%%%%%

\subsection{Agol-Wise's theorem and a quantitative result}\label{quant}

Hyperbolic $3$-manifolds have finite degree covers with positive Betti number, which was proved by Agol, based on the work of Kahn-Markovic, Wise, and many others. In fact more is true; we have the following properties (see \cite[Theorem~9.2]{Ago} for the closed case and \cite[Theorem~1.3]{Cooper_Long_Reid_VH} for the case with toroidal boundary), in order of strength. 

\begin{theorem}\label{theorem_virtbetti}
  Let $M$ be a hyperbolic $3$-manifold with finite volume. Then there exists:
  \begin{enumerate}[font=\normalfont]
  \item \label{pos_virtbetti} a finite cover $\widehat{M}\to M$ so that $b_1\left(\widehat{M}\right) > 0$;

  \item \label{infinite_virtbetti} for any $r \ge 1$, a finite cover $\widehat{M}\to M$ so that $b_1\left(\widehat{M}\right) \ge r$;

  \item \label{large} a finite cover $\widehat{M}\to M$ so that $\pi_1(M)$ surjects onto a non-abelian free group.
  \end{enumerate}
\end{theorem}

We note that (\ref{pos_virtbetti}) can be used to give a proof of Proposition~\ref{proposition_hypcase} which is similar but simpler than that of Proposition \ref{lattices}: if $H^1(\Gamma)$ contains a class $\phi$ of infinite order then the systole of the index $n$ subgroup $\Gamma_n:=\phi^{-1}(n\ZZ)$ stays bounded as $n \to +\infty$. As $\Gamma$ is residually finite, there exists a sequence of subgroups $\Gamma_m'$ whose systoles tend to infinity. For $m$ large enough it thus follows from Mostow rigidity that the subgroups $\Gamma_{[\Gamma : \Gamma_m']}$ and $\Gamma_m'$ are both of index $[\Gamma:\Gamma_m']$ and not isomorphic to each other. 

\medskip

In fact a much stronger quantitative result holds. The strongest result (\ref{large}), together with an argument due to Lubotzky and Belolipetsky-Gelander-Lubotzky-Shalev (\cite{BGLM}, \cite[Section 5.2]{BGLS}) shows that the number $e_d(M)$ of pairwise non-isometric covers of a hyperbolic manifold of finite volume satisfies 
\[
\limsup_{d \to +\infty} \frac{\log e_d(M)}{d\log(d)} > 0.
\]

%%%%%%%%%%%%%%%%%%%%%%%%%%%%%%%%%%%%%%%%%%%%%%%%%%%%%%%%%%%%

\subsection{Regular covers of hyperbolic manifolds}

In this subsection we use (\ref{infinite_virtbetti}), with $r = 2$, to prove the following result about regular covers. 

\begin{proposition} \label{main_normal}
  Let $\Gamma$ be the fundamental group of a hyperbolic $3$-manifold (complete of finite volume). Then there exists sequences $c_n, d_n \to +\infty$ such that for each $n$ we can find at least $c_n$ normal subgroups of index $d_n$ in $\Gamma$, which are pairwise non-isomorphic.
\end{proposition}

\noindent An important step in the proof is the following special case.  

\begin{proposition} \label{special_norm}
  Let $\Gamma$ be a lattice in a simple Lie group $G$, not isogenous to $\mathrm{PGL}_2(\RR)$. Assume that $b_1(\Gamma) \ge 2$ and, for all $n \ge 1$, let $c_n$ be the maximal number of pairwise non-isomorphic normal subgroups $\Gamma' \lhd \Gamma$ with $\Gamma / \Gamma' \cong \ZZ / n$. Then
  \[
  \liminf_{n \to +\infty} \frac{c_n} n > 0. 
  \]
\end{proposition}

We note that the hypothesis on \( \Gamma \) implies that \( G \) is of real rank 1, and in fact isogenous to one of \( \mathrm{SO}(n, 1) \) or \( \mathrm{SU}(n, 1) \), as lattices in higher-rank simple Lie groups have property (T) and hence finite abelianisation, as do those in \( \mathrm{Sp}(n, 1) \) and the exceptional rank 1 group \( F_4^{-20} \). 

%%%%%%%%%%%%%%%%%%%%%%%%%%%%%%

\subsubsection{Remarks}

\begin{itemize}
\item Proposition \ref{special_norm} shows that when $b_1(\Gamma) \ge 2$, for any large enough $n$ there exists a pair of non-isomorphic normal subgroups of index $n$ within $\Gamma$. The conclusion of Proposition~\ref{main_normal} is much weaker, and we do not know whether in general there are non-homeomorphic normal covers for every degree in a subset of \( \NN \) of natural density one. Note that in general this cannot be true of every degree---for example if \( M \) is a homology sphere then it cannot have regular covers of any prime degree.

\item We still have some control over the density of the sequence $d_n$ in Proposition \ref{main_normal}: it follows from the proof that we have $d_n \ll n^M$ where $M = r!$, with $r$ the smallest index of a normal subgroup with $b_1 \ge 2$. 

\item Moreover, the proof of Proposition \ref{main_normal} shows that we can take $c_n \gg d_n^{1/e - \eps}$ for all $\eps > 0$. 
  
\item The only ingredient specific to dimension 3 in the proof of Proposition \ref{main_normal} is property~(\ref{infinite_virtbetti}). We note that this property holds for many lattices in higher dimensions as well (in particular all known lattices in $\mathrm{SO}(n,1)$ in even dimensions), and for some complex hyperbolic lattices (see for example \cite{Marshall_unitary}). 

\item In \cite{Zimmermann}, Zimmermann produces a similar set of subgroups in $\Gamma$. In that construction, the quotients are isomorphic to $\ZZ/p^{n-i}\ZZ \oplus \ZZ/p^i\ZZ$ where $p$ is some large prime. Moreover, the number of subgroups in Zimmermann's construction is sublinear as a function of their index.

%\item With more sophisticated tools and assuming the stronger hypothesis that $\Gamma$ surjects onto a non-abelian free group (but no essential addition to the argument) we can replace the sequence $\ZZ/n$ in Proposition \ref{special_norm} by any sequence of finite simple groups (we get a slightly weaker quantitative conclusion). It seems likely that the scheme of proof can adapt to many more sequence of finite groups, under suitable assumptions on $\Gamma$ (depending on the sequence). 
\end{itemize}

%%%%%%%%%%%%%%%%%%%%%%%%%%%%%%

\subsubsection{Proof of Proposition \ref{special_norm}}

Let $\varphi_2(n)$ be the number of surjective homomorphisms from $(\ZZ/n)^2$ to $\ZZ/n$ and let $\varphi(n)$ be Euler's totient. Observe that $(1,0)\mapsto a, (0,1)\mapsto *$ defines $\varphi(n)\times n$ surjective homomorphisms from $(\ZZ/n)^2$ to $\ZZ/n$, where $(a,n)=1$ and $*$ is any arbitrary element of $\Z/n$. Similarly, there are $n\times \varphi(n)$ surjective homomorphisms defined by $(1,0)\mapsto *, (0,1)\mapsto b$, where $(b,n)=1$ and $*$ is an arbitrary element of $\Z/n$. Counted without repetition, we have produced 
\[
\varphi(n)\times n + n \times \varphi(n) - \varphi(n) \times\varphi(n) = (2n - \varphi(n))\varphi(n)
\] 
surjective homomorphisms from $(\ZZ/n)^2$ to $\ZZ/n$. Thus, $\varphi_2(n) \geq (2n - \varphi(n))\varphi(n)$.

%Let $\varphi_2(n)$ be the number of surjective morphisms from $(\ZZ/n)^2$ to $\ZZ/n$, which is equal to the number of primitive elements in $(\ZZ/n)^2$. Let $\varphi(n)$ be Euler's totient function, then there are at least
%\[
%\varphi(n)\cdot n + n \cdot \varphi(n) - \varphi(n) \cdot\varphi(n) = (2n - \varphi(n))\varphi(n)
%\]
%primitive elements: the first (resp. second) summand on the left counts elements whose first (resp. second) coordinate is coprime to $n$, and the third takes their intersection away. So we have:
%\[
%\varphi_2(n) \geq (2n - \varphi(n))\varphi(n) 
%\]

Let $h_n$ be the number of surjective morphisms from $\Gamma$ to $\ZZ/n$. Since $\Gamma$ surjects onto $(\ZZ/n)^2$ we have that $h_n \ge \varphi_2(n)$. Since two surjective morphisms $\pi_1, \pi_2 : \Gamma \to Q$ have the same kernel if and only if there exists an automorphism $\psi$ of $Q$ such that $\pi_2 = \psi \circ \pi_1$, and the number of automorphisms of $\ZZ/n$ equals $\varphi(n)$ and hence
\[
\frac{h_n} {\card{\mathrm{Aut}(\ZZ/n)}} \ge \frac{\varphi_2(n)}{\varphi(n)} \ge 2n - \varphi(n) \ge n
\]
we get that there are pairwise distinct normal subgroups $A_1, \ldots, A_n \le \Gamma$ such that $\Gamma / A_j \cong \ZZ / n$ for $1 \le j \le n$. 

By Mostow rigidity, we have that $c_n$ is at least the maximal number of $A_j$ which are pairwise non-conjugate in $G$. We want to prove that there exists $b> 0$ depending only on $\Gamma$ such that for every $n$ at most $b$ among the $A_j$s can be conjugated to each other, which implies that $c_n \ge n/b$ and finally the conclusion of Proposition~\ref{special_norm}. For this we use a refinement of the arguments of \cite{BGLM} and \cite{BGLS} that we mentioned in the previous subsection.

\medskip

First we deal with the case where $\Gamma$ is non-arithmetic: then an immediate and well-known consequence of Margulis' commensurator criterion for arithmeticity\footnote{The criterion \cite[Theorem (B) in Chapter IX]{Margulis_book} states that $\Gamma$ has finite index in its commensurator $\Omega$; since any lattice commensurable to $\Gamma$ commensurates $\Gamma$, it has to be contained in $\Omega$ and the claim follows.} is that there is a unique maximal lattice $\Omega \subset G$ in the commensurability class of $\Gamma$, which is equal to the commensurator of $\Gamma$. Thus any $g \in G$ which conjugates two $A_j$s must belong to $\Omega$, and since the $A_j$s are normal in $\Gamma$ each has at most $b = |\Omega / \Gamma|$ conjugates among them. 

\medskip

Now assume $\Gamma$ is arithmetic. By definition of arithmeticity there exists a semisimple algebraic group $\mathbf G$ defined over $\ZZ$ such that $\Gamma \subset \mathbf G(\ZZ)$ with finite index. For $p$ a rational prime let $\ZZ_p$ denote the $p$-adic integers. Then a {\em congruence subgroup} of $\Gamma$ is a subgroup of the form $\Gamma \cap U$ where $U$ is a finite index (equivalently, open) subgroup in $\prod_p \mathbf G(\ZZ_p)$. If $\Lambda$ is a finite index subgroup in $\Gamma$ we denote by $\ccl{\Lambda}$ the {\em congruence closure} of a subgroup $\Lambda \subset \Gamma$: this is the smallest congruence subgroup of $\mathbf G(\QQ)$ containing $\Lambda$; explicitely the congruence closure of $\Lambda$ is equal to $\Gamma \cap V$ where $V$ is the closure in $\prod_p \mathbf G(\ZZ_p)$ of $\Lambda$. 

\begin{lemma} \label{abelian_cong}
  Let $\Gamma$ be an arithmetic group. There exists finitely many congruence subgroups $\Gamma_1, \ldots, \Gamma_m$ with the following property: if $\Lambda$ is a finite index normal subgroup in $\Gamma$ such that $\Gamma / \Lambda$ is abelian then $\ccl{\Lambda}$ is equal to one of the $\Gamma_i$. 
\end{lemma}

\begin{proof}
  If $\Gamma / \Lambda$ is abelian then so is $\Gamma / \ccl{\Lambda}$ so that it suffices to show that there are only finitely many congruence subgroups $\Delta \le \Gamma$ such that $\Gamma / \Delta$ is abelian. Let $\Gamma'$ be the derived subgroup of $\Gamma$, which is a Zariski-dense in $\mathbf G$ (since $\mathbf G$ does not have abelian quotients). By Nori-Weisfeiler strong approximation \cite{Nori,Weisfeiler} it follows that the closure of $\Gamma'$ in $\prod_p \mathbf G(\ZZ_p)$ has finite index in that of $\Gamma$. This means that at most finitely many congruence subgroups of $\Gamma$ contain $\Gamma'$, which is the statement we wanted to prove. 
\end{proof}

The commensurator of $\Gamma$ is equal to (the image in $G$ of) $\mathbf G(\QQ)$ and we have that $\ccl{g\Lambda g^{-1}} = g\ccl{\Lambda} g^{-1}$ for all $g \in \mathbf G(\QQ)$. It follows that if two subgroups of $\Gamma$ are conjugate to each other an element conjugating them must belong to $\mathbf G(\QQ)$ and conjugate their congruence closures to each other as well: in particular, if the latter are equal then the element must belong to its normaliser. Let \( \Gamma_1, \ldots, \Gamma_m \) be given by the lemma and let \( i \) be an index such that \( \Gamma_i \) contains \( n' \ge n/m \) of the \( A_j \)s, and assume for notational ease that those are \( A_1, \ldots, A_k \). It follows from the above that for any $n$, any element conjugating two of the $A_1, \ldots, A_k$ must belong to the normaliser \( \Omega_i \) of $\Gamma_i$. Thus, as the \( A_j \) are normal in \( \Gamma \), the maximal number of conjugates among \( A_1, \ldots, A_k \) is $c = |\Omega_i/(\Gamma_i \cap \Gamma)|$. In conclusion, we have shown that we can find at least \( k/C = n/(cm) \) among the \( A_j \) that are pairwise not conjugate in \( \mathrm{PGL}_2(\CC) \), hence our claim follows (with \( b = cm \)). 
  
\begin{proof}[Proof of Proposition \ref{main_normal}]
  
Let $\Delta$ be a lattice in $\mathrm{PGL}_2(\CC)$. By Theorem~\ref{theorem_virtbetti}(\ref{infinite_virtbetti}) there exists a finite index normal subgroup $\Gamma \triangleleft \Delta$ with $b_1(\Gamma) \ge 2$, so that we may apply Proposition \ref{special_norm} to $\Gamma$.
% We will also need to assume that \( \Gamma \) is normal in \( \Delta \). 
Let $a_1, \ldots, a_r$ be representatives for the left cosets of $\Gamma$ in $\Delta$. Let $n \ge 1$ and $B_1, \ldots, B_{c_n}$ the subgroups obtained in Proposition \ref{special_norm}. Then since $B_j \triangleleft \Gamma$ we get, for all $1 \le j \le c_n$, that: 
\[
C_j = \bigcap_{i=1}^r a_i B_j a_i^{-1}
\]
is normal in $\Delta$.

We recall that if $A$ is a permutation group of degree $r$ (that is, a subgroup of the symmetric group $\sym_r$) and $B$ any group, the \emph{wreath product} $A \wr B$ is the semidirect product $A \rtimes B^r$ where $A$ acts on $B^r$ by permuting indices. 

\begin{lemma} \label{bigger_quotient}
There exist $r,l\in\NN$, depending only on $\Gamma$ and $\Delta$, such that for each $1\le j\le c_n$, there exists a finite abelian group $Q_j$ so that
\[(\ZZ/n)^l \twoheadrightarrow Q_j \]
and
\[\Delta/C_j \hookrightarrow \sym_r \wr  Q_j.\]
%  Each group $\Delta/\Gamma_j$ is isomorphic to a subgroup of $\sym_r \wr (\ZZ/n)^l$ (where \( l \) is independent of \( n \)). 
\end{lemma}

\begin{proof}
  Let $\rho$ be the morphism $\Delta \to \sym(\Delta /C_j)$ associated to the left-translation action. It respects the decomposition into left $\Gamma$-cosets and hence it has image inside $\sym(\Delta/\Gamma) \wr \sym(\Gamma/C_j)$. Moreover the stabiliser of a block has its image in a conjugate of the image of the action of $\Gamma$ on $\Gamma/C_j$.

  It remains to see that \( \Gamma/C_j \) is abelian of exponent \( n \) (we can then take \( l \) to be the rank of \( H_1(\Gamma) \otimes \ZZ/n \)). To do so we need only remark that the subgroup  \([\Gamma,\Gamma]\cdot\Gamma^n < \Gamma \) generated by commutators and \( n \)th powers 
is characteristic in \( \Gamma \) and contained in \( B_j \), hence it is also contained in \( C_j \). Since
\[[\Gamma,\Gamma]\cdot\Gamma^n = \ker\left( \Gamma \to H_1(\Gamma) \otimes \ZZ/n \right)\]
this implies that
\[H_1(\Gamma)\otimes \ZZ/n \twoheadrightarrow \Gamma/C_j\] 
and hence that the image of $\Delta/C_j$ in the second factor in the wreath product is a quotient of $H_1(\Gamma)\otimes \ZZ/n$.
\end{proof}

By the same argument as in the proof of the previous proposition we can eliminate some of the $C_j$s so that at least $b_n = c_n/a$ (where $a$ depends only on $\Gamma$) are pairwise non-isomorphic. Indeed, if $\Delta$ is non-arithmetic then the same argument applies verbatim, while if $\Delta$ is arithmetic we have to show that for $C_j \triangleleft \Delta$ with $\Delta/C_j \hookrightarrow \sym_r \wr Q_j$ there are only finitely many possibilities for the congruence closures of $C_j$ in $\Delta$. To do this, we only need to note that this is true of the congruence closures of the $C_j$ in 
\[\Delta_1:= \ker\left(\Gamma \to \sym_r\right).\]
Indeed, since all the $\Delta_1/C_j$ are abelian, this follows from Lemma \ref{abelian_cong}. Moreover, since these closures contain those in $\Delta$ it is also true of the latter. 

So we may assume that $C_1, \ldots, C_{b_n}$ are pairwise non-conjugate, with $b_n \ge c_n/a \ge n/a'$ for some $a' \ge 1$ independent of $n$. A priori the $C_j$ have different indices in $\Delta$. But by Lemma \ref{bigger_quotient} the orders $|\Delta/G_j|$ all divide $r! \cdot n^{lr}$. Let $\delta(N)$ denote the number of divisors of a positive integer $N$, then using the classical estimate that $\forall \eps > 0$ there exists a constant $K_\varepsilon$ so that
\[
\delta(N) \leq K_\varepsilon\cdot N^\eps
\]
for all $N\in\NN$ (see e.g. \cite[Proposition 7.12]{dKL}) and the fact that $b_n \gg (r! \cdot n^{rl})^{1/(rl)}$ we see that
\[
b_n / \delta(r! \cdot d_n^r) \xrightarrow[n \to+\infty]{} +\infty.
\]
Hence by the pigeonhole principle we see that as $n \to +\infty$ an unbounded number of the $C_j$ have the same index in $\Delta$. This finishes the proof of Proposition~ \ref{main_normal}. 
\end{proof}

%%%%%%%%%%%%%%%%%%%%%%%%%%%%%%%%%%%%%%%%%%%%%%%%%
%		E U C L I D E A N   M A N I F O L D S
%%%%%%%%%%%%%%%%%%%%%%%%%%%%%%%%%%%%%%%%%%%%%%%%%
\section{Euclidean manifolds}

The general argument we will present for Seifert fibered manifolds (Proposition \ref{prp_seif})  will not hold for manifolds that are finitely covered by either the $3$-torus or the $3$-sphere. Seifert fibered manifolds that are finitely covered by $T^3$ are exactly the closed \emph{Euclidean manifolds} (see eg. \cite{Sco}). In this section, we prove the following proposition.

\begin{proposition}~\label{prop:euclidean}
The only closed orientable exceptional Euclidean $3$-manifold is the $3$-torus~$T^3$. 
\end{proposition}

In fact we will prove a result about a more general class of groups. A group \( \Gamma \) is called a {\em crystallographic group} if it acts properly discontinuously and cocompactly on a finite-dimensional vector space; we refer to \cite{Auslander_cristal} for an overview. It then has to preserve a Euclidean metric on this space, so another definition is to say that crystallographic groups are lattices in isometry groups of Euclidean spaces.

In particular torsion-free crystallographic groups are the fundamental groups of finite-volume flat Riemannian manifolds, also called Euclidean manifolds. These are classified up to dimension~4. In dimension 2 there are only the torus and Klein bottle. It has been known since the 1930s that there are only ten closed $3$-manifolds which are covered by the $3$-torus, among which four are non-orientable~\cite{Nowacki,HW}. 
The ten closed Euclidean $3$-manifolds can be explicitly constructed~\cite{platycosms}, and as such concrete geometric arguments can be used to show that none of these but the $3$-torus are exceptional. However it is perhaps simpler to use a more algebraic argument to prove the following more general result. 

\begin{proposition} \label{proposition_euccase}
  Let \( \Gamma \) be a crystallographic group. Then \( \Gamma \) is exceptional if and only if it is free abelian. 
\end{proposition}

\begin{proof}
  Let \( E \) be a Euclidean vector space such that \( \Gamma \) is a lattice in \( \isom(E) \cong \OG(E) \ltimes E \). Let \( \pi \) be the map from \( \Gamma \) to \( \OG(E) \). Then \( T = \ker(\pi) \) is a free abelian group of rank \( \dim(E) \) by Bieberbach's theorem. In the sequel we will identify it with a subgroup of \( E \). From now on we will assume that \( \Gamma \) is not free abelian, so \( \pi \) has nontrivial image. 

  Let \( \Pi = \pi(\Gamma) \), let \( p \) be a prime dividing the order \( |\Pi| \) and let \( m = |\Pi|/p \). Choose an element \( \sigma \in \Pi \) which has order \( p \).

  \begin{lemma}
    Let \( E_f = \ker(\sigma - \id ) \). Then
    \[
    T' = (T \cap E_f) \oplus (T \cap E_f^\bot)
    \]
    is a finite index subgroup in \( T \).
  \end{lemma}

  \begin{proof}
    It suffices to show that \( T' \) has the same rank as \( T \), in other words, that \( T \cap E_f \) and \( T \cap E_f^\bot \) are of respective rank \( \dim(E_f) \) and \( \dim(E_f^\bot) \).

    Let \( \omega = (\id + \sigma + \cdots + \sigma^{p-1})|_T \). We recall that \( \sigma T \subset T \): if \( \gamma \in \Gamma \) is any element projecting to \( \sigma \) and \( v \in T \), a quick computation shows that \( \gamma v\gamma^{-1} = \sigma v \) and hence \( \sigma v \in T \). On the other hand, \( \omega \) is equal to \( p \) times the orthogonal projection onto \( E_f \) and we have that \( \omega T \subset (T \cap E_f) \). So the image is discrete and has rank at most \( \dim(E_f) \), and the kernel has rank at most \( \dim(E_f^\bot) = \dim(E) - \dim(E_f) \). Thus, both inequalities must be equalities and this finishes the proof.
  \end{proof}

  Let 
  \[
  A = \pi^{-1}\langle\sigma\rangle.
  \]
  Since \( \sigma \) preserves \( E_f, E_f^\bot \), and \( T \) we have that \( \sigma T' = T' \) and hence, using the semi-direct product structure, that the subgroup \( T' \) is normal in \( A \). Likewise, we see that the subgroup
  \[
  T'' = (T \cap E_f) \oplus p(T \cap E_f^\bot)
  \]
  is also normal in \( A \) and as we have \( \sigma \not= \id \) we have \( \dim(E_f^\bot) = r > 0 \) and we get that
  \[
  |T/T''| = p^r |T/T'|.
  \]
  Now choose any subgroup \( L \) of index \( p^{r-1} |T/T'| \) in \( T \), \( \gamma \in \Gamma \) such that \( \pi(\gamma) = \sigma \) and let
  \[
  B = \langle\gamma\rangle T''
  \]
  so that \( B/T'' = \ZZ/p \) (because \( \gamma^p \in E_f \cap T \)). It follows that
  \begin{align*}
    |\Gamma/B| &= m/p \cdot |A/B| = m/p \cdot p^r|T/T'| \\
    &= mp^{r-1}|T/T'| = m|T/L| = |\Gamma/L|
  \end{align*}
  which finishes the proof: indeed, \( B \) is not abelian while \( L \) is and it follows that they cannot be isomorphic. 
\end{proof}

\begin{proof}[Proof of Proposition~\ref{prop:euclidean}]
Recall that $3$-manifolds with abelian fundamental group are well-understood; see e.g.~\cite[Table~2]{AFW}. Observe that if a closed Euclidean $3$-manifold $M$ has free abelian $\pi_1(M)$, then $\pi_1(M)\cong \mathbb{Z}^3$ and $M\cong T^3$, which we know is exceptional from Proposition~\ref{prop:specialones}.
\end{proof}

%%%%%%%%%%%%%%%%%%%%%%%%%%%%%%%%%%%%%%%%%%%%%%%%%
%		S P H E R I C A L   M A N I F O L D S
%%%%%%%%%%%%%%%%%%%%%%%%%%%%%%%%%%%%%%%%%%%%%%%%%
\section{Spherical manifolds}\label{sec_spherical}

Seifert fibered manifolds finitely covered by the $3$-sphere, namely the \emph{spherical manifolds}, also require a different proof than the general case. We will soon completely classify the exceptional spherical $3$-manifolds (Proposition~\ref{prp_spherical}). However, we will first need some notation. 

It is known that spherical $3$-manifolds are exactly the quotients of $\Sphere^3$ by finite groups that act by isometries~\cite{Sco}. These quotients of $\Sphere^3$ have been classified by Hopf \cite[Section 2]{Hopf} (see also \cite[p.\ 12]{AFW} and~\cite[Theorem~2]{Milnor:1957-1}) as follows (note that the group $Q_{4n}$ (in the notation of  \cite[p.\ 12]{AFW}) is isomorphic to $D_{2n}$ when $n$ is odd). 

\begin{theorem}\label{thm_sphericalclassification} The fundamental group of a spherical $3$-manifold is of exactly one of the following forms:
\begin{itemize}
\item The trivial group,
\item $Q_{8n} := \langle x,y|\; x^2=(xy)^2 = y^{2n} \rangle$, for $n\geq 1$,
\item the binary octahedral group: $P_{48} := \langle x,y|\; x^2 = (xy)^3 = y^4, \; x^4 =1 \rangle $,
\item the binary icosahedral group: $P_{120} := \langle x,y|\; x^2=(xy)^3=y^5,\;x^4=1\rangle $,
\item $D_{2^m(2n+1)} := \langle x,y |\; x^{2^m}=1,\;y^{2n+1} =1,\; xyx^{-1} = y^{-1} \rangle$, 
for $m\geq 2$ and $n\geq 1$,
\item $P_{8\cdot 3^m}' := \langle x,y,z|\;x^2=(xy)^2=y^2,\; zxz^{-1}=y,\;zyz^{-1} = xy,\; z^{3^m}=1 \rangle$, for $m\geq 1$,
\item the direct product of any of the above groups with a cyclic group of relatively prime order.
\end{itemize}
\end{theorem}

The subscripts in the notation for the groups above always denote their order. We are now ready to state the main result of this section. 

\begin{proposition}\label{prp_spherical} A spherical manifold is exceptional if and only if its fundamental group is of one of the following forms:
\begin{itemize}
\item The trivial group,
\item $Q_8$,
\item $P_{120}$,
\item $D_{2^m(2n+1)}$ for $m\geq 2$ and $n\geq 1$,
\item $P_{8\cdot 3^m}'$ for $m\geq 1$,
\item the direct product of any of the above groups with a cyclic group of relatively prime order.
\end{itemize}
\end{proposition}

Before giving the proof, we gather a few relevant facts. Recall that the fundamental group determines a spherical manifold unless it is cyclic and non-trivial (see \cite[p.\ 133]{Orlik:1972-1}). This fact, for the larger class of closed, irreducible $3$-manifolds, yields the following lemma (see~\cite[Theorem~2.1.2]{AFW}). 

\begin{lemma}\label{lemma:fundgp} Let $M$ be a closed, orientable, irreducible $3$-manifold. Let $G$ and $H$ be finite index subgroups of $\pi_1(M)$ and let $\widehat{M_G}$ and $\widehat{M_H}$ be covers of $M$ that correspond to $G$ and $H$, respectively. Suppose $G$ is not a finite cyclic group, then $G$ and $H$ are isomorphic if and only if $\widehat{M_G}$ and $\widehat{M_H}$ are homeomorphic.
\end{lemma}

The case of spherical manifolds with cyclic fundamental groups, namely, lens spaces, is more subtle. We recall the following well known result of Reidemeister \cite{Reidemeister:1935-1}.
\begin{theorem}\label{theorem:Reidemeister} Let $L(p, q)$ and $L(p, q')$ be two lens spaces. Then $L(p, q)$ and $L(p, q')$ are homeomorphic if and only if $q' \equiv \pm q^{
\pm1} \mod p$.
\end{theorem}

\noindent We are now ready to prove Proposition~\ref{prp_spherical}.

\begin{proof}[Proof of Proposition~\ref{prp_spherical}] 
For each of the manifolds listed in Theorem~\ref{thm_sphericalclassification}, we will follow one of the following two strategies. 
\begin{itemize}
\item In order to show that a given manifold is not exceptional, we will show that its fundamental group has two non-isomorphic subgroups with the same index. 
\item In contrast, in order to show that a given manifold is exceptional, we will first show that its fundamental group has a unique isomorphism type of subgroup with any fixed index. By Lemma~\ref{lemma:fundgp}, this implies that it only remains to consider the case when the subgroups of a given index are all isomorphic to a fixed finite cyclic group. In this case, we will show that corresponding covers are homeomorphic, either by using Theorem~\ref{theorem:Reidemeister} or by showing that these subgroups are conjugate to each other. 
\end{itemize}

We divide our proof into multiple lemmata. In what follows we will often tacitly identify a manifold with its fundamental group. 

\noindent First we note:
\begin{lemma}
Any spherical $3$-manifold with cyclic fundamental group (namely, a lens space) is exceptional.
\end{lemma}

\begin{proof} 
This follows since a cyclic group contains at most one subgroup of a given index. 
\end{proof}

\begin{lemma} The spherical $3$-manifold with fundamental group $P_{48}$ is not exceptional while that with fundamental group $P_{120}$ is exceptional.
\end{lemma}

\begin{proof}
The proper subgroups of $P_{48}$ and $P_{120}$ are well known (see e.g.~\cite[Appendix]{Lima-Guaschi:2013-1}). The group $P_{48}$ has $\Z/8$ and $Q_8$ as proper subgroups, and hence, the spherical manifold with fundamental group $P_{48}$ is not exceptional. The group $P_{120}$ has at most one isomorphism type of subgroup of any given index, and those of order $2,3,4,5,6$, and $10$ are isomorphic to finite cyclic groups. Note that by Theorem~\ref{theorem:Reidemeister} there is a unique $3$-manifold with fundamental group with order $2,3,4,$ or $6$. Any two proper subgroups of $P_{120}$ of order $5$ are Sylow $5$-subgroups and hence conjugate to each other. Also, it is known that any order two element of $P_{120}$ is contained in the center. Since the order $5$ subgroups are conjugate, this implies that the order $10$ subgroups are also conjugate to one another. Thus, we see that the spherical manifold with fundamental group $P_{120}$ (namely the Poincar\'{e} homology sphere) is exceptional. 
\end{proof}

All that remains are the three infinite sequences and products with cyclic groups of coprime order. We start  with the groups $Q_{8n}$:

\begin{lemma}
The spherical $3$-manifold with fundamental group $Q_{8n}$ is exceptional if and only if $n=1$.
\end{lemma}

\begin{proof}
The group $Q_8$ is the quaternion group, the proper subgroups of which are all cyclic with either order $2$ or $4$. Using a similar argument as above, one can show that the spherical manifold with fundamental group $Q_8$ is exceptional. Next, we prove that for $n>1$, the group $Q_{8n}$ is not exceptional. In particular, we will show that these groups have two non-isomorphic subgroups of index $2$. Let $N_1 = \langle\langle y\rangle\rangle$ and $N_2=\langle\langle x\rangle\rangle$ be the subgroups normally generated by $y$ and $x$ respectively. It is easy to verify that both $N_1$ and $N_2$ are subgroups of index $2$. Note that $N_1$ is cyclic since $xyx^{-1}=y^{-1}$ in $Q_{8n}$. Since $N_1$ has index $2$, the order of $N_1$ is $4n$ and in particular, $y$ has order $4n$ in $Q_{8n}$. On the other hand, we will show that $N_2$ is non-abelian. Using the $x=yxy$ relation, it is easy to see that $yxy^{-1}\cdot x^{-1} = y^{2} \in N_2$. Now suppose that $y^2xy^{-2}=x$. Since $x^2=(xy)^2$, and hence $x=yxy$, we see that $y^4x=x$, and thus, that the order of $y$ is at most $4$, which contradicts our previous observation, since $n>1$. Thus, the groups $Q_{8n}$ for $n>1$ are not exceptional. 
\end{proof}

\noindent For the dihedral groups we have:
\begin{lemma}
The spherical $3$-manifold with fundamental group $D_{2^m(2n+1)}$ are exceptional for all $m\geq 2$ and $n\geq 1$.
\end{lemma}

\begin{proof}
We will invoke some Sylow theory. First, note that the subgroup generated by $x\in D_{2^m(2n+1)}$ is a Sylow $2$-subgroup of $D_{2^m(2n+1)}$ and is isomorphic to $\Z/2^m$. It is also easy to show that the abelianization of $D_{2^m(2n+1)}$ is isomorphic to $\Z/2^m$, and is generated by the image of $x$. Let $N_1=\langle\langle y\rangle\rangle$ be the subgroup of $D_{2^m(2n+1)}$ normally generated by $y$. Note that $N_1$ is cyclic, since $xyx^{-1} = y^{-1}$. Thus, we have the following exact sequence of groups 
\[1\to N_1\cong \Z/(2n+1) \to D_{2^m(2n+1)} \to \Z/2^m \cong \langle x\rangle \to 1 \]
corresponding to the abelianization map.  
Further, let $N_2 = \langle y, x^2 \rangle$. Since $x^2$ commutes with $y$, we have the following exact sequence
\[1\to N_2 \cong \Z/(2n+1) \times \Z/2^{m-1} \to D_{2^m(2n+1)} \to \Z/2\cong\{0,x\} \to 1.\]
Clearly, $N_2$ is cyclic of order $2^m(2n+1)$ and is generated by $yx^2$.

Let $\Gamma\leq D_{2^m(2n+1)}$ be a subgroup and $\phi$ be the restriction to $\Gamma$ of the homomorphism $D_{2^m(2n+1)} \to \Z/2$ in the last sequence. We first consider the case where $\phi$ is a surjection. Then $\Gamma$ must also surject onto $\Z/2^m$ in the abelianization map, since every map from $D_{2^m(2n+1)}$ to an abelian group factors through the abelianization. There are two options -- either $\card{\Gamma} = 2^m$ or $\card{\Gamma}>2^m$. In the first case, $\Gamma$ is cyclic and conjugate to $\langle x \rangle$ by Sylow's theorem. If $\card{\Gamma}>2^m$, we now show that $\Gamma$ must be isomorphic to $D_{2^m(2n'+1)}$ for some $n'\leq n$. 

First, note that due to the relation $xyx^{-1}=y^{-1}$, any element of $D_{2^m(2n+1)}$ can be written as $y^ix^j$ for some $0\leq i\leq 2n$ and $0\leq j\leq 2^m-1$. 
Recall that $N_2=\langle yx^2\rangle$. Then $\ker \phi = \Gamma\cap N_2=\Gamma \cap \langle yx^2\rangle$ and thus, $\ker \phi=\langle (yx^2)^d\rangle=\langle y^dx^{2d}\rangle$ for some $d\geq 1$. Since $[\Gamma : \langle y^dx^{2d} \rangle ] = 2$ and $\lvert \langle y^dx^{2d} \rangle \rvert = \frac{2^{m-1}(2n+1)}{d}$, we have $\card{\Gamma} = \frac{2n+1}{d}\cdot2^m$. Since $\card{\Gamma}>2^m$, we see that $d<2n+1$. 

Since $\Gamma$ surjects onto $\Z/2^m$, $\card{\Gamma}$ is divisible by $2^m$. As a result, $d$ divides $2n+1$. This implies that $d$ is odd and $\frac{2n+1}{d}$ is an integer. 
Note that $(y^dx^{2d})^\frac{2n+1}{d}=x^{2(2n+1)}$. Since $(2n+1, 2^{m-1})=1$, $x^{2(2n+1)}$ generates $\langle x^2\rangle$. As a result, $x^2\in \langle y^dx^{2d}\rangle$, and thus, $y^d\in\langle y^dx^{2d}\rangle$. 
Next, let $y^ix^j$ be an element of $\Gamma$ such that $y^ix^j\notin \langle y^dx^{2d}\rangle$. If no such element exists, then $\Gamma \leq N_2$ which is a contradiction. We show that $j$ is odd, as follows. Suppose for the sake of contradiction that $j$ is even. Since $x^2\in \langle y^dx^{2d}\rangle$ and $y^ix^j\notin \langle y^dx^{2d}\rangle$, $y^i\notin \langle y^dx^{2d}\rangle$. Since $y^d\in \langle y^dx^{2d}\rangle$, this implies that $i$ is not a multiple of $d$. In particular, $\gcd(i,d)<d$. Moreover, since $y^i\in\Gamma$ and $y^d\in\Gamma$, $y^{\gcd(i,d)}\in \Gamma$. 
However, note that the order of $y^{\gcd(i,d)}$ is $\frac{2n+1}{\gcd(i,d)}$ since the order of $y$ in $D_{2^m\cdot (2n+1)}$ is $2n+1$. Then, $\frac{2n+1}{\gcd(i,d)}$ divides $\card{\Gamma}=\frac{2n+1}{d}\cdot2^m$, which implies that $d$ divides $2^m\cdot \gcd(i,d)$. Since $d$ is odd, it must divide $\gcd(i,d)<d$ which is a contradiction. Thus, $j$ must be odd; denote $j$ by $2k+1$ for some $k$. 

We will now complete the proof by showing that the subgroup generated by $y^d$ and $y^ix^{2k+1}$ is equal to $\Gamma$ and isomorphic to $D_{2^m(2n'+1)}$ where $2n'+1=\frac{2n+1}{d}$. We have the following identities
\begin{align*}
(y^ix^{2k+1})^{2^m}&=(y^ix^{2k+1}y^ix^{2k+1})\cdots (y^ix^{2k+1}y^ix^{2k+1})\\
	&=x^{2k+1}\cdots x^{2k+1}\\\
	&=(x^{2k+1})^{2^m}\\
	&=1,\\
(y^d)^\frac{2n+1}{d}&=y^{2n+1}=1,\\
(y^ix^{2k+1})y^d(y^ix^{2k+1})^{-1}&=y^ix^{2k+1}y^dx^{-2k-1}y^{-i}\\
	&=y^ixy^dx^{-1}y^{-i}\\
	&=y^iy^{-d}dy^{-i}\\
	&=y^{-d},
\end{align*}
where we have used the facts that $x^2y^dx^{-2}=y^d$, $xyx^{-1}=y^{-1}$, and $yxy=x$. It is easy to check that any element of $\langle y^d, y^ix^{2k+1}\rangle$ can be uniquely expressed in the form $(y^d)^{i'}(x^{2k+1})^{j'}$, where $0\leq i'\leq \frac{2n+1}{d}-1$ and $0\leq j'\leq 2^m-1$ and thus, $\lvert \langle y^d, y^ix^{2k+1}\rangle\rvert =\card{\Gamma}=\frac{2n+1}{d}\cdot2^m$ which completes the argument.

When $\phi$ is not surjective, it must be the zero map and thus, $\Gamma$ is a subgroup of the cyclic subgroup $N_2 \cong  \Z/((2n+1)2^{m-1})$ which implies that it is cyclic.

Returning to the group $D_{2^m(2n+1)}$, we have now shown that subgroups of a given fixed index are isomorphic. More precisely, subgroups are of the form $D_{2^m(2n'+1)}$ for $n'\leq n$, or cyclic groups with order either $2^m$, or a factor of $(2n+1)2^{m-1}$. The latter arose as subgroups of a cyclic group and thus, occur exactly once. We saw earlier that any subgroup of order $2^m$ is a Sylow $2$-subgroup, and thus such subgroups are conjugate to one another. This concludes that spherical manifolds with fundamental group $D_{2^m(2n+1)}$ are exceptional for $m\geq 2$ and $n\geq 1$.
\end{proof}

\noindent For the sequence $P_{8\cdot 3^m}'$ we have:

\begin{lemma}
The spherical $3$-manifold with fundamental group $P_{8\cdot 3^m}'$ is exceptional for all $m\geq 1$.
\end{lemma}

\begin{proof} First note that $P_{8\cdot 3^m}'^{ab}\cong \Z/3^m$. Since $Q_8\cong \langle x, y\rangle < P_{8\cdot 3^m}'$ lies in the kernel of the abelianization map and $\card{Q_8}=8$, it must actually coincide with this kernel. Thus this copy of $Q_8$ is a normal subgroup and as such is the unique Sylow-2 subgroup of $P'_{8\cdot 3^m}$. We obtain the following exact sequence
\[1\to Q_8 \to P_{8\cdot 3^m}'\to \Z/3^m \to 1\]
corresponding to the abelianization map. 
Now let $\Gamma \leq P_{8\cdot 3^m}$. If the image of $\Gamma$ in the abelianization map is trivial, it needs to lie in $Q_8$, hence it is either $Q_8$, $\Z/2$, or $\Z/4$. 

For the case when the image of $\Gamma$ is non-trivial, we first show that $\Gamma$ is either a finite cyclic group or $Q_8\times \Z/3^j$ for some $0<j\leq m$. Since the order of $\Gamma$ divides the order of $P_{8\cdot 3^m}'$, we see that $\card{\Gamma}=2^i\cdot3^{j}$ for some $0\leq i \leq3, 0\leq {j} \leq m$. Since $\Gamma$ has non-trivial image in $P_{8\cdot 3^m}'^{ab}$, we see that ${j}\neq0$. With these restrictions, we see from the list in Theorem~\ref{thm_sphericalclassification} that the only possible groups are $D_{4\cdot3^j}$, $D_{8\cdot 3^j}$, $P'_{8\cdot 3^j}$, $Q_{8\cdot 3^j}$, $Q_8\times \Z/3^j$, $\Z/3^j$, $\Z/{2\cdot 3^j}$, $\Z/{4\cdot 3^j}$, and $\Z/{8\cdot 3^j}$, where $0<j\leq m$. 

It is straightforward to see that the abelianization of $D_{2^m\cdot (2n+1)}$ is $\Z/2^m$, and the abelianization of $Q_{8\cdot 3^j}$ is $\Z/2\times \Z/2$. Neither of these can surject onto $\Z/3^j$. Next, suppose that $8$ divides the order of $\Gamma$. Then $\Gamma$ has a unique Sylow-2 subgroup, namely $Q_8$, where the uniqueness follows from normality. Since the unique Sylow-2 subgroup of $\Z/8\cdot 3^j$ is cyclic, we see that $\Gamma\ncong  \Z/8\cdot 3^j$. So, if $8$ divides the order of $\Gamma$, we see that $\Gamma$ is either isomorphic to $P'_{8\cdot 3^j}$ or $Q_8\times \Z/3^j$. In these cases $\Gamma$ has a Sylow-3 subgroup, denoted by $\Gamma_3$. Recall that there is a Sylow-3 subgroup of $P'_{8\cdot 3^m}$ which is a copy of $\Z/3^m$ generated by $z$. Since any Sylow-3 subgroup of $\Gamma$ must be contained in some Sylow-3 subgroup of $P'_{8\cdot 3^m}$ and Sylow-3 subgroups of a given group are conjugate, there is some $g\in P'_{8\cdot 3^m}$ such that $g^{-1}z^{3^{m-j}}g$ generates $\Gamma_3 \leq \Gamma$, and thus, $z^{3^{m-j}}\in g\Gamma g^{-1}$ generates $g\Gamma_3 g^{-1}$. Note that the subgroup $g\Gamma g^{-1}\leq P'_{8\cdot 3^m}$ has order divisible by $8$, so as before, its Sylow-2 subgroup coincides  with the Sylow-2 subgroup of $P'_{8\cdot 3^m}$, namely the copy of $Q_8$ generated by $x$ and $y$. Since $z^3$ commutes with $<x,y>$, we see that $<x,y,z^{3^{m-j}}> \cong Q_8\times \Z/3^j \leq g\Gamma g^{-1}$ for $j\neq m$ and in fact, $\Gamma \cong Q_8\times \Z/3^j$ due to cardinality. We have thus reduced the possibilities for $\Gamma$ to $\Z/3^j$, $\Z/{2\cdot 3^j}$, $\Z/{4\cdot 3^j}$, and $Q_8\times \Z/3^j$, where $0<j\leq m$ when the image of $\Gamma$ is non-trivial.

Thus, the possible subgroups of $P'_{8\cdot 3^m}$ are of the form $\ZZ/2, \ZZ/4, Q_8, \ZZ/3^j, \ZZ/2\cdot3^j, \Z/4\cdot3^{j}, Q_8\times \Z/3^j$, where $0<j\leq m$. Again, no two distinct isomorphism types of subgroups have the same order, therefore we only need to consider the case of cyclic subgroups. By Theorem~\ref{theorem:Reidemeister} there is a unique $3$-manifold with fundamental group $\ZZ/2$ or with fundamental group $\ZZ/4$. For $\ZZ/3^j$, we know that this group is contained some Sylow-$3$ subgroup of $P_{8\cdot 3^m}'$, which is cyclic. Hence any two subgroups with order $3^j$ are conjugate to each other, and thus, correspond to isomorphic covering spaces. Suppose $\Gamma \cong \ZZ/2\cdot3^j\cong \ZZ/2\times \ZZ/3^j$. Note there is a Sylow-$3$ subgroup of $\Gamma$ which is contained in some Sylow-$3$ subgroup of $P_{8\cdot 3^m}'$ which is cyclic and conjugate to $\langle z \rangle$, as we saw earlier. Hence there is some $g \in P_{8\cdot 3^m}'$ such that $g^{-1}z^{3^{m-j}}g$ is contained in $\Gamma$. Similarly, any Sylow-2 subgroup of $\Gamma$ corresponds to a subgroup of order two within the copy of $Q_8$ in $P'_{8\cdot 3^m}$. There is a unique such subgroup, generated by $x^2$. Thus, since $\Gamma$ is cyclic, we see that $\Gamma = \langle x^2, g^{-1}z^{3^{m-j}}g \rangle$. Note that any two such subgroups are conjugate to each other since $x^2$ is central. 
Lastly, suppose $\Gamma \cong \Z/4\cdot3^{j}$. Let $\Gamma_2$ be a Sylow-$2$ subgroup of $\Gamma$. Then it is contained in $Q_8$, namely the Sylow-2 subgroup of $P'_{8\cdot 3^m}$, and consequently, it is either $\langle x \rangle, \langle y \rangle$, or $\langle xy \rangle$. Further, let $\Gamma_3$ be a Sylow-$3$ subgroup of $\Gamma$. Then again $g^{-1}z^{3^{m-j}}g$ is contained in $\Gamma$ for some $g \in P_{8\cdot 3^m}'$. We have now seen that $g\Gamma g^{-1}$ is either $\langle x, z^{3^{m-j}} \rangle, \langle y, z^{3^{m-j}} \rangle,$ or  $\langle xy, z^{3^{m-j}}\rangle$ and any two such subgroups are conjugate to each other since $z^{3^{m-j}}$ is central, and $x$, $y$ and $xy$ are conjugates. This completes the proof.
\end{proof}

\noindent Finally, we need to consider direct products with cyclic groups:

\begin{lemma} Let $G$ be a group from the statement of Theorem~\ref{thm_sphericalclassification} and $C$ a finite cyclic group so that $\gcd(\card{G},\card{C})=1$. Then $G\times C$ is exceptional if and only if $G$ is exceptional.
\end{lemma}

\begin{proof}
Subgroups of direct products of finite groups with relatively prime order are direct products of subgroups in the factors. Hence, since the orders of the groups in the direct product need to be relatively prime and cyclic groups are exceptional, taking a direct product with a cyclic group preserves being exceptional or not. 
\end{proof}

\noindent We have now addressed each case in Theorem~\ref{thm_sphericalclassification} and thus, our proof is completed. 
\end{proof}
%%%%%%%%%%%%%%%%%%%%%%%%%%%%%%%%%%%%%%%%%%%%%%%%%
%		T H E   S E I F E R T   F I B E R E D   C A S E
%%%%%%%%%%%%%%%%%%%%%%%%%%%%%%%%%%%%%%%%%%%%%%%%%
\section{The general Seifert fibered case}\label{sect:SF}

\noindent In this section, we prove the following result. 

\begin{proposition}\label{prop:seifert-fibered-for-leitfaden}
Closed orientable Seifert fibered $3$-manifolds, other than those finitely covered by $T^3$, $S^1\times S^2$ or $S^3$, are not exceptional.
\end{proposition}
Note that along with Propositions~\ref{prop:euclidean} and~\ref{prp_spherical}, this shows that closed Seifert fibered $3$-manifolds are not exceptional.

Our proof is based on the following proposition, which can for instance be found in \cite[p.\ 52 (C.10)]{AFW}:
\begin{proposition} Let $M$ be a closed Seifert fibered $3$-manifold. There exists a finite cover $\widehat{M}\to M$ so that $\widehat{M}$ is an $\Sphere^1$-bundle over a closed orientable surface.
\end{proposition}

In order to find distinct finite covers of a Seifert fibered manifold, it thus suffices to find distinct finite covers of $\Sphere^1$-bundles over closed  orientable surfaces. Intuitively, the way we produce these is to take covers both in the $\Sphere^1$-direction and the surface direction.  To make this idea precise, we will use the Euler number $e(\pi)$ of our circle bundles to distinguish covers (see \cite[p. 427, 436]{Sco} for a definition). In particular, we will use the following property of Euler numbers (see for instance \cite[Lemma 3.5]{Sco}). 

\begin{lemma}\label{lem_eulernumber} Let $d\in \NN$ and
\[\Sphere^1\to M \stackrel{\pi}{\to} \Sigma\]
be an $\Sphere^1$-bundle over a closed oriented surface $\Sigma$, such that $M$ is orientable. Moreover, let $\widehat{M}$ be a degree $d$ finite cover of $M$, so that $\widehat{M}$ is the total space of the following bundle
\[\Sphere^1\to M \stackrel{\widehat{\pi}}{\to} \Sigma.\]
Suppose the induced circle and surface covers $\Sphere^1\to\Sphere^1$ and $\widehat{\Sigma}\rightarrow \Sigma$ have degrees $m$ and $\ell$ respectively, then $\ell m=d$ and
\[e(\widehat{\pi}) = \frac{\ell}{m} e(\pi). \]
\end{lemma}

\noindent We now prove the following proposition. This, together with the fact that the only orientable $S^1$-bundle over $S^2$ is  $S^1\times S^2$, will complete the proof of Proposition~\ref{prop:seifert-fibered-for-leitfaden}.

\begin{proposition}\label{prp_seif}Let $\Sigma$ be a closed orientable surface that is not a sphere and let $M$ be an $\Sphere^1$-bundle over $\Sigma$. Then $M$ is exceptional if and only if $M$ is the trivial $\Sphere^1$-bundle over the $2$-torus.
\end{proposition}

\begin{proof}
Our first claim is that the map $\pi_1(\Sphere^1) \to \pi_1(M)$ is injective. This follows from the long exact sequence in homotopy of the fibration
\[\ldots \to \pi_2(\Sigma) \to \pi_1(\Sphere^1) \to \pi_1(M) \to \pi_1(\Sigma)\to \ldots.\]
Our assumption on the genus of $\Sigma$ implies that $\pi_2(\Sigma) = \{e\}$ and hence that the map $\pi_1(\Sphere^1) \to \pi_1(M)$ is injective.

First we assume the bundle is non-trivial. Let $t\in \pi_1(\Sphere^1)$ denote a generator. Residual finiteness of $3$-manifold groups (see Theorem~\ref{theorem_resfin}) implies that we can find a finite group $G$ and a surjection $\varphi:\pi_1(M) \to G$ so that
\[ \varphi(t) \neq e.\]
Let us denote the induced degree $d=\card{G}$ cover by $\widehat{M}\to M$. Since $t$ is mapped to a non-trivial element, the induced circle cover is non-trivial. Then Lemma \ref{lem_eulernumber} tells us that the induced $\Sphere^1$-bundle $\Sphere^1\to \widehat{M} \stackrel{\widehat{\pi}}{\to} \widehat{\Sigma}$ satisfies
\[ \abs{e(\widehat{\pi})} < d\cdot \abs{e(\pi)}. \]

To build the second cover, take any degree $d$ surface cover $\widetilde{\Sigma}\to\Sigma$ (these exist for any $d$) and pull back the $\Sphere^1$-bundle. This gives rise to a degree $d$ cover $\widetilde{M}\to M$, that has the structure of a $\Sphere^1$-bundle $\Sphere^1\to \widetilde{M} \stackrel{\widetilde{\pi}}{\to} \widetilde{\Sigma}$. Applying Lemma \ref{lem_eulernumber} again, we obtain
\[ \abs{e(\widetilde{\pi})} = d\cdot \abs{e(\pi)}. \]

Since our bundle is non-trivial, we have $e(\pi)\neq 0$. It can now for instance be extracted from the Gysin sequence that if $\Sphere^1\to N\to \Sigma$ is a circle bundle with euler number $e\neq 0$, then
\[H_1(N,\ZZ) \cong \ZZ^{2g} \oplus \ZZ/e
\] 
where $g$ denotes the genus of $\Sigma$. In particular this implies that the absolute value of $e$ is an invariant of the total space and not just the circle bundle. That in turn means that $\widehat{M}$ and $\widetilde{M}$ are not homeomorphic.

Finally, we have to deal with the trivial bundle, that is, $M \cong \Sigma\times \Sphere^1$. In this case $e(\pi)=0$
and
\[H_1(M;\ZZ) \cong  \ZZ^{2g+1}. \] 
Now surface and circle covers 
\[\widehat{\Sigma}\to \Sigma \;\;\text{and}\;\;\Sphere^1\to\Sphere^1\]
of the same degree induce two covers 
\[\widehat{M} \cong \widehat{\Sigma}\times \Sphere^1\to M \;\;\text{and}\;\;\widetilde{M} \cong \Sigma\times \Sphere^1\to M\]
of the same degree. If $g >1$, then $\widehat{\Sigma}$ has strictly greater genus than $\Sigma$, so $\widehat{M}$ is not homeomorphic to $\widetilde{M}$.
\end{proof}

%%%%%%%%%%%%%%%%%%%%%%%%%%%%%%%%%%%%%%%%%%%%%%%%%
%		S O L V   M A N I F O L D S
%%%%%%%%%%%%%%%%%%%%%%%%%%%%%%%%%%%%%%%%%%%%%%%%%
\section{Sol manifolds}

\noindent In this section, we prove the following proposition. 

\begin{proposition}\label{prp_sol} 
Sol $3$-manifolds are not exceptional.
\end{proposition}

\begin{proof} Every orientable Sol manifold $M$ is finitely covered by a $2$-torus bundle over $\Sphere^1$ with Anosov monodromy $\varphi\in \mathrm{MCG}(\Torus^2) \cong \mathrm{SL}_2(\ZZ)$ (see for instance \cite[Theorem 1.8.2]{AFW}). Recall this the monodromy is called Anosov if the top-eigenvalue $\lambda_\varphi$ of $\varphi$ as an $\mathrm{SL}_2(\ZZ)$-matrix satisfies $\abs{\lambda_\varphi} > 1$. By Lemma \ref{lem:nonexceptional-cover}, we may assume $M$ is a $2$-torus bundle over $S^1$ with Anosov monodromy $\varphi$. Let $\lambda_\varphi$ be the leading eigenvalue. 

First we remind the reader of the well known fact that, as opposed to the case of hyperbolic mapping tori, the modulus of the eigenvalue $\lambda_\varphi$ is a topological invariant. Indeed, we have $b_1(M) = 1$. Each fibration $\pi:M\to S^1$ with connected fibers and monodromy $\psi$ induces a primitive non-torsion cohomology class $[\psi]\in H^1(M;\ZZ)$ and conversely each such cohomology class determines the fibration up to isotopy. The latter fact implies that the top eigenvalue $\lambda_\psi$ of the monodromy $\psi$ depends only on $[\psi]$. The former observation implies that the only fibered classes are $[\varphi]$ and $-[\varphi]$. Since 
\[\abs{\lambda_{[-\varphi]}} = \abs{\lambda_{[\varphi]}},\]
$\abs{\lambda_\varphi}$ is indeed a topological invariant

In order to build two non-homeomorphic covers, we proceed as follows. First let 
\[\Torus^2\to\Torus^2\]
be a finite non-trivial characteristic cover. This means that $\varphi$ lifts to a map 
\[\widehat{\varphi}: \Torus^2\to \Torus^2\]
such that $\lambda_{\varphi} = \lambda_{\widehat{\varphi}}$. We obtain a cover
\[\widehat{M} \to M,\]
where 
\[\widehat{M} \cong \Torus^2 \times [0,1] / (x,0)\sim (\widehat{\varphi}(x),1).\]

Since $b_1(M) =1$, we can also take a finite cyclic cover 
\[ \widetilde{M}\to M\]
of the same degree, say $d\neq 1$. The monodromy $\widetilde{\varphi}$ of this cover satisfies
\[\abs{\lambda_{\widetilde{\varphi}}} = \abs{\lambda_\varphi}^d,\]
thus $\widehat{M}$ and $\widetilde{M}$ are not homeomorphic.
\end{proof}

%%%%%%%%%%%%%%%%%%%%%%%%%%%%%%%%%%%%%%%%%%%%%%%%%
%		N O N T R I V I A L  J S J   D E C O M P O S I T I O N
%%%%%%%%%%%%%%%%%%%%%%%%%%%%%%%%%%%%%%%%%%%%%%%%%
\section{Manifolds with non-trivial JSJ decompositions and non-trivial boundary}

\noindent 
In this section, we prove the following proposition.

\begin{proposition}\label{prp_JSJ}
Let $M$ be an orientable, irreducible $3$-manifold with empty or toroidal boundary such that either $M$ has a non-trivial JSJ decomposition or $\partial M$ is non-empty. Assume that $M$ is not homeomorphic to $S^1\times D^2$, $T^2\times I$ or the twisted $I$-bundle over the Klein bottle, and that $M$ is not a Sol manifold. Then $M$ is not exceptional. 
\end{proposition}

\noindent We start out with the following useful lemma.

\begin{lemma}\label{lem:extend-homomorphisms}
Let $M$ be an orientable, irreducible $3$-manifold with empty or toroidal boundary and let $N$ be a JSJ component of $M$. Then for any finite group $G$ and any surjective homomorphism $f\colon\pi_1(N)\twoheadrightarrow G$, there exist finite groups $K$ and $H$ and homomorphisms $g, g_1, g_2, g_3$ $($of the type shown in the diagram$)$, such that the following diagram commutes.
\[
\xymatrix@C1.4cm@R1.3cm{\pi_1(N)\ar@{->>}[d]^f\ar[r]^{i_*}\ar@{->>}[dr]^{g_3} &\pi_1(M)\ar@{->>}[dr]^g\\
G
&K\ar@{->>}[l]_-{g_1} \ar@{^(->}[r]^{g_2}& H.}
\]
\end{lemma}

Note in particular that the cover of $N$ induced by the map $g_3$ is a cover of the one induced by $f$. 

\begin{proof}
For closed manifolds this lemma is  an immediate consequence of \cite[Theorem~A]{WZ10}. As is explained in \cite[(C.35)]{AFW}, the statement also holds in the case that $M$ has non-empty toroidal boundary.
\end{proof}

\begin{proof}[Proof of Proposition~\ref{prp_JSJ}]
In this proof we use the following terminology. Given a 3-manifold $W$ with empty or toroidal boundary we refer to the union of the JSJ tori and the boundary tori of $W$ as the set of  \emph{characteristic tori} of $W$.
At the end of the upcoming proof we will have constructed  two index $d$ covering spaces $\widehat{M}$ and $\widetilde{M}$ of $M$ that we will distinguish by showing that they have unequal numbers of characteristic tori. 

We say that a $3$-manifold is \emph{tiny} if it is homeomorphic to $S^1\times D^2$, $T^2\times I$, or to the twisted $I$-bundle over the Klein bottle.
Throughout this proof we will use on several occasions the following preliminary remark: If $W$ is an orientable $3$-manifold that is 
not tiny, then it follows from the classification of 
 $3$-manifolds with virtually solvable fundamental group, see~\cite[Theorem~1.11.1]{AFW}, that no finite cover of $W$ is tiny.

By~\cite{Hem} (see also~\cite[C.10]{AFW}), we know that $M$ has a finite-sheeted cover $M'\rightarrow M$ such that each Seifert fibered JSJ component of $M'$ is an $S^1$-bundle over a compact orientable surface. 
Since $M$ is not a Sol manifold, neither is $M'$. By our preliminary remark, since $M$ is not tiny, neither is $M'$. These three latter facts, along with~\cite[Propositions 1.9.2 and 1.9.3]{AFW}, imply that this manifold $M'$ has the following useful property $(\dagger)$:
For any finite cover $\widehat{M'}\rightarrow M'$, the preimage of the JSJ decomposition of $M'$ is exactly the JSJ decomposition of $\widehat{M'}$. 
%This follows from along with the fact that $M'$ is not a Sol manifold, the fact that each Seifert fibered JSJ-component of $M'$ is an $S^1$-bundle and the fact that $M'$ is not tiny. 

Let $N'$ be a JSJ component of $M'$, where possibly $N'=M'$. By hypothesis, $\partial N'$ is non-empty. There exists a finite-sheeted cover $\overline{N'}\rightarrow N'$ such that the rank of the cokernel of the map $H_1(\partial \overline{N'})\rightarrow H_1(\overline{N'})$ is at least one, by~\cite[C.15, C.17]{AFW}. (Here
we used that $N'$ is not homeomorphic to $T^2\times I$, this follows from the fact that  $M'$ is not tiny and from our hypothesis that $M$ is not a Sol-manifold and from \cite[Proposition~1.6.2(3), 1.8.1, 1.10.1]{AFW}.)

The finite-sheeted cover $\overline{N'}\rightarrow N'$ corresponds to a finite index subgroup of $\pi_1(N')$. Recall that any finite index subgroup of a group contains a finite index normal subgroup, called its normal core; let the (finite index, regular) cover corresponding to the latter normal subgroup of $\pi_1(N')$ be denoted $\widehat{N'}\rightarrow N'$. By construction, $\widehat{N'}\rightarrow N'$ corresponds to the kernel of a surjective map $\pi_1(N')\twoheadrightarrow G$ for some finite group $G$, and by Lemma~\ref{lem:extend-homomorphisms}, we obtain the following commutative diagram:
\[
\xymatrix@C1.4cm@R1.3cm{\pi_1(N')\ar@{->>}[d]^f\ar[r]^{i_*}\ar@{->>}[dr]^{g_3} &\pi_1(M')\ar@{->>}[dr]^g\\
G \cong \pi_1(N')/\pi_1(\widehat{N'})
&K\ar@{->>}[l]_-{g_1} \ar@{^(->}[r]^{g_2}& H.}
\]

Let $M^*$ be the cover $M^*\rightarrow M'$ corresponding to the kernel of $g$. From Lemma~\ref{lem:extend-homomorphisms}, it follows that the induced cover of $N'$ corresponding to $g_3$ is a finite-sheeted cover of $\widehat{N'}$; call it $N^*$. Since $M^*$ is a finite-sheeted cover of $M'$ it follows from $(\dagger)$ that  $N^*$ is a JSJ component of $M^*$. Since the cover $N^*\rightarrow \overline{N'}$ is finite-sheeted it follows from an elementary argument, see~\cite[A.12]{AFW}, that the rank of the cokernel of the map $H_1(\partial N^*)\rightarrow H_1(N^*)$ is also at least one. 
Since $M^*$ is a finite-sheeted cover of $M$ it follows from  Lemma~\ref{lem:nonexceptional-cover} that it suffices to show  that $M^*$ is not exceptional. 

We have the following commutative diagram, where the horizontal sequences form the long exact sequence in singular homology for the pairs $(N^*, \partial N^*)$ and $(M^*, M^*\setminus \Int(N^*))$ and the vertical arrows are induced by inclusion. 
\begin{equation*}
\begin{tikzcd}
 H_1(\partial N^*) \arrow{r}{i_*}\arrow{d} &H_1(N^*) \arrow{r}\arrow{d}{k_*} &H_1(N^*, \partial N^*)\arrow{d} &\\
 H_1(M^*\setminus \Int(N^*)) \arrow{r}{j_*} &H_1(M^*) \arrow{r} &H_1(M^*, M^*\setminus \Int(N^*)).
\end{tikzcd}
\end{equation*}
We obtain an induced commutative diagram
\begin{equation*}
\begin{tikzcd}
\coker(i_*) \arrow[hookrightarrow]{r}\arrow{d}{\overline{k_*}} &H_1(N^*, \partial N^*)\arrow{d}\\	
\coker(j*) \arrow[hookrightarrow]{r} &H_1(M^*, M^*\setminus \Int(N^*)).
\end{tikzcd}
\end{equation*}
By a standard excision argument, we see that $H_1(M^*, M^*\setminus \Int(N^*))\cong H_1(N^*, \partial N^*)$; in other words, the rightmost vertical map above is an isomorphism. We see that $\overline{k_*}$ is injective, and thus, the rank of $\coker(j_*)$ is bounded below by the rank of $\coker(i_*)$ which by hypothesis is at least one. Thus, there is an epimorphism $\coker(j_*)\twoheadrightarrow \Z$. We can then define, for any $m>1$, 
\[
\begin{tikzcd}
f_m\colon\pi_1(M^*) \arrow[two heads]{r}{\operatorname{Ab}} &H_1(M^*)\arrow[two heads]{r} & \coker(j_*)\arrow[two heads]{r} &\ZZ \arrow[two heads]{r} &\Z/m,
\end{tikzcd}
\]
where the second and last maps are the canonical projections. 
Note that each characteristic torus of $M^*$ is contained in $M^*\setminus \Int(N^*)$, and thus, by our construction, the image under inclusion of the fundamental group of any characteristic torus of $M^*$ lies in the kernel of $f_m$ for all $m$. 

We are finally ready to construct two non-homeomorphic covers of $M^*$ with the same index. As mentioned in the beginning of the proof, we will do so by constructing two finite covers of $M^*$ with the same degree, but different number of characteristic tori. At this point we would like to recall that it follows from $(\dagger)$ for any finite cover $W$ of $M^*$ the characteristic tori of $W$ are the preimages of the characteristic tori of $M^*$.

Now let $n$ be the total number of characteristic tori of $M^*$. 
Let $T$ be  a characteristic torus of $M^*$. Since $M^*$ is not tiny (and in particular, $M^*$ is not homeomorphic to $S^1\times D^2$, and thus, the boundary tori of $M^*$ are $\pi_1$-injective) we can view $\pi_1(T)$ as a subgroup of $\pi_1(M)$.  Since $\pi_1(M^*)$ is residually finite (Theorem~\ref{theorem_resfin}), there is some finite index normal subgroup $J\lhd \pi_1(M^*)$ such that $\pi_1(T)$ is not contained in $J$. Let $d>1$ be the index of $J$ in $\pi_1(M^*)$, and let $\widehat{p}\colon \widehat{M}\rightarrow M^*$ be the $d$-sheeted cover of $M^*$ corresponding to $J$. Since $\pi_1(T)$ is not contained in $J$ we see that the preimage of $T$ has strictly fewer than $d$ components. Since the preimage under $\widehat{p}$ of the characteristic tori for $M^*$ gives the characteristic tori for $\widehat{M}$, we see that the number of characteristic tori in $\widehat{M}$ is strictly less than $d\cdot n$. In order to build a second cover, let $\widetilde{p}\colon \widetilde{M}\rightarrow M^*$ denote the cover corresponding to the kernel of $f_d\colon \pi_1(M^*)\rightarrow \Z/d$ constructed above. By construction, the index of $\widetilde{p}$ is $d$ and the number of characteristic tori in $\widetilde{M}$ is $d\cdot n$, since the image under inclusion of characteristic torus of $M^*$ lies in the kernel of $f_d$. We see that $\widehat{M}$ and $\widetilde{M}$ are index $d$ covers of $M^*$ but have an unequal number of characteristic tori, and thus are non-homeomorphic. \end{proof}

%%%%%%%%%%%%%%%%%%%%%%%%%%%%%%%%%%%%%%%%%%%%%%%%%
%		N O N P R I M E   M A N I F O L D S
%%%%%%%%%%%%%%%%%%%%%%%%%%%%%%%%%%%%%%%%%%%%%%%%%

\section{Non-prime and non-orientable $3$-manifolds}\label{sec:nonprime}

\subsection{Prime non-orientable $3$-manifolds}
Little further work is needed to completely characterize exceptional prime non-orientable $3$-manifolds with empty or toroidal boundary. Such a manifold is either the twisted $S^2$-bundle over $S^1$, or irreducible. We already saw that the former is exceptional (Proposition~\ref{prop:specialones}). For the latter case, note that if the orientable double cover of a non-orientable $3$-manifold $M$ is not exceptional, then $M$ is not exceptional. By our previous work, we only need to consider the irreducible non-orientable $3$-manifolds whose orientable double covers are $S^1\times S^2$, the $3$-torus $T^3$, $S^1\times D^2$, $S^1\times S^1\times [0,1]$. Here we have used that closed non-orientable $3$-manifolds have positive first Betti number. In particular, their fundamental groups are infinite~\cite[(E.3)]{AFW}.

Note that if a boundary torus for a manifold $M$ is compressible, $M$ has an $S^1 \times D^2$ summand. If $M$ is prime, $M$ must be homeomorphic to $S^1\times D^2$, which is of course not non-orientable. Thus, a prime non-orientable $3$-manifold with toroidal boundary, must have incompressible boundary. By~\cite[Lemma~2.1]{Swarup1973}, the fundamental group of the orientable double cover of an irreducible non-orientable $3$-manifold $M$ with incompressible boundary is free, which would be the case if the double cover is $S^1\times S^2$ or $S^1\times D^2$, if and only if $M$ is a homotopy $\RR\textup{P}^2\times S^1$. Such a manifold has first homology group $\Z\oplus \Z/2$. Thus, it has two double covers, only one of which is orientable. As a result, it is not exceptional. 

It remains to consider the  non-orientable $3$-manifolds with empty or toroidal boundary, whose orientable double cover is $T^3$ or $T^2\times I$. First suppose that $M$ is a $3$-manifold that is covered by $T^3$.
It follows from~\cite[Theorem~2.1]{Meeks-Scott} that $M$ is Euclidean.
This case is then addressed by Proposition~\ref{proposition_euccase}.  
Secondly, if $M$ is covered by $T^2\times I$, then a doubling argument shows, using the above, that $M$ is either homeomorphic to the Klein bottle times an interval or to the M\"obius band times $S^1$. In either case $M$ admits two $2$-fold coverings, only one of which is orientable.
Thus $M$ is not exceptional.
Thus, we have established the following proposition.

\begin{proposition}\label{prop:non-orientable-prime}
Let $M$ be a prime, non-orientable $3$-manifold with empty or toroidal boundary. Then $M$ is exceptional if and only if it is homeomorphic to $S^1\widetilde{\times} S^2$.
\end{proposition}

\subsection{Non-prime $3$-manifolds}
Finally, we consider non-prime $3$-manifolds. Recall that a prime decomposition for a $3$-manifold $M$ is said to be \emph{normal} if there is no $S^1\times S^2$ factor when $M$ is non-orientable. It is well-known that every  $3$-manifold has a unique normal prime decomposition (see, for instance~\cite[Theorem~3.15 and~3.21]{Hempel-book}). Note that we do not assume that the manifold is closed or orientable. 

\noindent Below we give a general structure for covering spaces of non-prime  $3$-manifolds. 

\begin{proposition}\label{prop:cover-of-connected-sum}
Let $M\cong M_1\#\cdots \#M_k$ be a normal prime decomposition for a  $3$-manifold $M$. Then if $\widehat{M}$ is a cover of $M$ with index $d$, then 
\[
\widehat{M}\cong \left(\widehat{M_{11}}\# \cdots \#\widehat{M_{1i_1}}\right)\, \# \cdots \#\, \left(\widehat{M_{k1}}\# \cdots\#\widehat{M_{ki_k}}\right)\#\, \ell S,
\]
where $S$ denotes $S^1\times S^2$ or $S^1\widetilde{\times} S^2$ whenever $\widehat{M}$ is orientable or non-orientable respectively, and $\widehat{M_j}:= \left(\widehat{M_{j1}}\sqcup \cdots \sqcup\widehat{M_{ji_1}}\right)$ is a cover of $M_j$ with index $d$. Moreover, 
\[
(k-1)\cdot d = \sum_{j=1}^k i_j -1 +\ell.
\]
In addition, any $\widehat{M}$ of this form is a cover of $M$ with index $d$. 
\end{proposition}
Since any cover of a prime $3$-manifold is prime, the above gives a prime decomposition of $\widehat{M}$ when we ignore $S^3$ summands. 

\begin{proof}Let $p\colon \widehat{M}\to M$ be the cover. Write $M$ as $M_1\setminus B_1 \cup \cdots \cup M_k\setminus B_k$ where each $B_j\subset M_j$ is an open ball. Then each  restricted map 
\[
p\vert_{p^{-1}(M_j\setminus B_j)}\colon p^{-1}(M_j\setminus B_j)\rightarrow M_j\setminus B_j
\]
is a covering map. Gluing balls along the lifts of the connected sum spheres gives rise to the covers $\widehat{M_j}$. 

It is clear that the cover $\widehat{M}$ is built from the collection $\bigsqcup \widehat{M_j}$ by identifying spheres in pairs, arising as lifts of the connected sum spheres. When the sphere pairs are in distinct connected components, we obtain a connected sum. When they lie in the same connected component, we obtain an $S$ summand. Note that $S$ can be chosen to be $S^1\widetilde{\times} S^2$ whenever $\widehat{M}$ is non-orientable since $N\# S^1\times S^2 \cong N\# S^1\widetilde{\times} S^2$ whenever $N$ is a non-orientable $3$-manifold. The relationship between $k$, $d$, $\ell$, and $i_j$ follows from the construction. 

When $M$ is oriented, each $\widehat{M_j}$ inherits an orientation and it is easy to see that $\widehat{M}$ is a connected sum of oriented manifolds.

For $\widehat{M}$ of the given form, an index $d$ covering map $\widehat{M}\to M$ can be constructed by gluing together the individual covering maps. 
\end{proof}

\noindent For non-prime $3$-manifolds, we have the following result. 

\begin{proposition}\label{prop:nonprime}
Let $M$ be a non-prime $3$-manifold with empty or toroidal boundary. Then $M$ is exceptional if and only if it is homeomorphic to $k\cdot S^1\times S^2$, for some $k\geq 2$.
\end{proposition}

\begin{proof}
First, we show that $k\cdot S^1\times S^2$ is exceptional. Note that the only cover of $S^1\times S^2$ is itself. Then, we see immediately from Proposition~\ref{prop:cover-of-connected-sum} that any degree $d$ cover of $k\cdot S^1\times S^2$ is homeomorphic to $((k-1)d +1)\cdot S^1\times S^2$. 

Next, we show that any manifold which is not of the form $k\cdot S^1\times S^2$ is not exceptional. Let $M\cong M_1\# \cdots \# M_k$ be a normal prime decomposition for $M$. By hypothesis, $k\geq 2$.

As a preliminary step, we observe that if $M$ has a single prime summand which is not exceptional, it is itself not exceptional. Since both $S^2$-bundles over $S^1$ are exceptional, such a prime summand must be irreducible. Without loss of generality, assume that $M_1$ is irreducible and not exceptional. Then, there exist non-homeomorphic covers $\widehat{M_1}$ and $\widetilde{M_1}$ of $M_1$, both with index $d$, for some $1<d<\infty$. Construct the covers $\widehat{M}\cong \widehat{M_1}\# d(M_2\# \cdots \#M_k)$ and  $\widetilde{M}\cong \widetilde{M_1}\# d(M_2\# \cdots \#M_k)$ of $M$. Note that both have index $d$. Since $\widehat{M_1}$ and $\widetilde{M_1}$ are both irreducible, they appear in the prime decomposition of $\widehat{M}$ and $\widetilde{M}$ respectively. By the uniqueness of normal prime decompositions, $\widehat{M}$ and $\widetilde {M}$ are not homeomorphic. 

Thus, we only need to consider the case where each $M_i$ is itself exceptional. First we consider the case where $M$ is orientable. Since $M$ is not of the form $k\cdot S^1\times S^2$, we can assume, without loss of generality, that $M_1$ be an exceptional manifold other than $S^1\times S^2$, that there exists a cover $\widehat{M_1}\to M_1$ of index $d_1$ and a cover $\widehat{M_k}\to M_k$ of index $d_k$, such that $d_1\leq d_k$. 

We now build two covers of $M$ with index $d_k$ as follows. Let 
\[
\widehat{M}= d_k(M_1\# \cdots \# M_{k-1}) \# \widehat{M_k},
\]
and let 
\[
\widetilde{M}= \left(\widehat{M_1}\# (d_k-d_1)M_1\right) \# (d_1-1)S^2\times S^1 \# d_k(M_2\# \cdots M_{k-1})\# \widehat{M_k}.
\]
Suppose that $\widehat{M}\cong \widetilde{M}$. Then we see that $M_1\cong S^1\times S^2$, which is a contradiction.

Next, consider the case when $M$ is non-orientable. Suppose first that there is at least one irreducible prime summand in the given normal prime decomposition of $M$. Without loss of generality, we can assume that this is $M_1$. Since no non-orientable irreducible $3$-manifold is exceptional, we see that $M_1$ is orientable. Let $N$ denote the manifold $M_2\# \cdots \# M_k$. Since $M\cong M_1\# N$ is non-orientable, we see that $N$ is non-orientable. Let $\widehat{N}$ be the orientable double cover of $N$ and construct the orientable double cover $2M_1\#\widehat{N}$ of $M$. Since $M_1$ is irreducible, it is in particular not $S^1\times S^2$. By our argument in the previous paragraph, $2M_1\#\widehat{N}$ has two non-homeomorphic covers of the same index, showing that $M$ is not exceptional. 

It only remains to consider the case where $M$ is non-orientable but has no irreducible prime summands. Since we have a normal prime decomposition, this implies $M$ is of the form $k\cdot S^1\widetilde{\times} S^2$ where $k\geq 2$. Note that in this case $H_1(M)\cong \mathbb{Z}^k$ where $k\geq 2$. Then, we see that there are $2^k-1$ connected double covers, only one of which is orientable, which completes the proof. 
\end{proof}

\noindent We observe that we have established the following proposition. 
\begin{proposition}\label{prop:nonorientable}
The only non-orientable exceptional $3$-manifold with empty or toroidal boundary is the twisted $S^2$-bundle over $S^1$. 
\end{proposition}

%%%%%%%%%%%%%%%%%%%%%%%%%%%%%%%%%%%%%%%%%%%%%%%%%
%		R E F E R E N C E S
%%%%%%%%%%%%%%%%%%%%%%%%%%%%%%%%%%%%%%%%%%%%%%%%%

\bibliographystyle{alpha}
\bibliography{bib}

\end{document}